\documentclass{amsart}

\usepackage[utf8]{inputenc}
\usepackage{amsmath,amssymb,mathtools,amsthm}
\usepackage{enumerate,multirow,soul,xcolor,tikz,hyperref}
\usetikzlibrary{arrows}

\theoremstyle{plain}
\newtheorem{thm}{Theorem}[section]
\newtheorem{propn}[thm]{Proposition}
\newtheorem{lem}[thm]{Lemma}
\newtheorem{cor}[thm]{Corollary}

\theoremstyle{definition}
\newtheorem{defn}[thm]{Definition}
\newtheorem{eg}[thm]{Example}
\newtheorem{q}[thm]{Question}

\theoremstyle{remark}
\newtheorem{rmk}[thm]{Remark}

\newcommand{\A}{\mathbf{A}}
\newcommand{\C}{\mathbf{C}}
\newcommand{\F}{\mathbf{F}}
\renewcommand{\P}{\mathbf{P}}
\newcommand{\Q}{\mathbf{Q}}
\newcommand{\R}{\mathbf{R}}
\newcommand{\Z}{\mathbf{Z}}

\let\le\leqslant
\let\ge\geqslant
\let\into\hookrightarrow
\let\bar\overline
\renewcommand{\tilde}{\widetilde}

\newcommand{\an}{\mathrm{an}}
\DeclareMathOperator{\End}{End}
\DeclareMathOperator{\Fix}{Fix}
\DeclareMathOperator{\lcm}{lcm}
\DeclareMathOperator{\Nm}{Nm}
\newcommand{\Num}{\mathcal{N}}
\DeclareMathOperator{\ord}{ord}
\DeclareMathOperator{\Per}{Per}
\DeclareMathOperator{\PGL}{PGL}
\DeclareMathOperator{\Poly}{Poly}
\DeclareMathOperator{\PrePer}{PrePer}
\DeclareMathOperator{\Rat}{Rat}
\DeclareMathOperator{\Res}{Res}
\DeclareMathOperator{\rk}{rk}
\DeclareMathOperator{\sgn}{sgn}
\newcommand{\vis}{\mathrm{vis}}
\DeclareMathOperator{\vol}{vol}

\definecolor{my_red}{RGB}{199,69,65}
\definecolor{my_green}{RGB}{76,151,88} 
\definecolor{my_blue}{RGB}{66,126,186}
\definecolor{my_purple}{RGB}{112,85,175}

\allowdisplaybreaks

\numberwithin{equation}{section} 

\begin{document}

\title[Distribution of preperiodic points]{Distribution of preperiodic points in one-parameter families of rational maps}
\author{Matt Olechnowicz}
\address{Department of Mathematics \& Statistics, Concordia University}
\email{matt.olechnowicz@concordia.ca}
\date{\today}

\begin{abstract}
Let $f_t$ be a one-parameter family of rational maps 
defined over a number field $K$.
We show that for all $t$ outside of a set of natural density zero, 
every $K$-rational preperiodic point of $f_t$ is the specialization of some $K(T)$-rational preperiodic point of $f$.
Assuming a weak form of the Uniform Boundedness Conjecture, 
we also 
calculate the average number of $K$-rational preperiodic points of $f$, 
giving some examples where this holds unconditionally.
To illustrate the theory, 
we give new estimates on the average number of preperiodic points 
for the quadratic family $f_t(z) = z^2 + t$ over the field of rational numbers.
\end{abstract}

\maketitle

\section{Introduction}

Let $K$ be a number field of degree $n \ge 1$ 
and let $\phi : \P^N \to \P^N$ be a morphism of degree $d \ge 2$. 
A point $P \in \P^N(K)$ is called \emph{preperiodic} if the sequence 
\[P, \, \phi(P), \, \phi^2(P), \ldots\]
is eventually periodic.
By Northcott's theorem \cite[Theorem 3]{Northcott},
the set 
\[\PrePer(\phi, K) = \{P \in \P^N(K) : P \text{ is preperiodic under } \phi\}\]
of $K$-rational preperiodic points of $\phi$
is finite,
and Morton and Silverman have conjectured \cite[p.~100]{MortonSilverman1994}
that its cardinality is uniformly bounded 
in terms of $d$, $N$, and $n$ only:
\begin{equation} \label{eq:UBC}
C(d, N, n)
:= \! \! \sup_{\substack{[K : \Q] \le n \\ \phi \in \End^N_d(K)}} \! \#{\PrePer}(\phi, K) < \infty.
\end{equation}
This is the Uniform Boundedness Conjecture (UBC).\footnote{The variety $\End^N_d$ is the parameter space of degree-$d$ endomorphisms of $\P^N$.}
Recently, Looper \cite[Theorem 1.2]{Looper} proved \eqref{eq:UBC} for polynomial maps of $\P^1$ assuming a generalization of the $abc$ conjecture.
At present, the best available (unconditional) bounds on $\#{\PrePer}(\phi, K)$ depend on the coefficients of $\phi$ in some nontrivial way (e.g.,~Benedetto \cite[Theorem 7.1]{Benedetto}, Troncoso \cite[Corollary 1.3(c)]{Troncoso}).

Inspired by recent advances 
in arithmetic statistics, 
one might wonder whether the Uniform Boundedness Conjecture is true ``on average''.
In this paper, 
we focus on 
one-parameter families $f_t \in K(z)$
of rational functions of degree $d \ge 2$
over a fixed number field $K$.
    By this we mean 
    the collection $\{f_t(z)\}_{t \in K}$ of specializations 
	of a single rational function $f \in K(T)(z)$,
	so that
	\[f_t(z) = \frac{a_d(t) z^d + \ldots + a_0(t)}{b_d(t) z^d + \ldots + b_0(t)} \in K(z)\]
	as the parameter $t$ varies in $K$ (see Section \ref{sec:domain_of_defn}).
Let $U$ be the domain of the rational map $\P^1 \dashrightarrow \Rat_d = \End^1_d$ induced by sending $t$ to $f_t$, 
and let $H : \P^1(K) \to \R_{\ge 1}$ denote the absolute multiplicative height.
Our goal is to estimate the total number of $K$-rational preperiodic points of the family, i.e.,~the quantity
\begin{equation} \label{eq:our_focus}
    \mathcal{A}(X) = \! \sum_{\substack{t \in U(K) \\ H(t) \le X}} \! \#{\PrePer}(f_t, K),
\end{equation}
in relation to the number of terms $\Num(U(K), X) = \Num(K, X) + O(1)$ as $X \to \infty$.
Note that $\Num(K, X) \sim c_K X^{2n}$ by Schanuel's theorem (see Section \ref{sec:def_heights}).

What is the main term in \eqref{eq:our_focus}?
Each of the sets $\PrePer(f_t, K)$ 
contains the image of $\PrePer(f, K(T))$ 
under the specialization map $\P^1(K(T)) \to \P^1(K)$.
Using Baker's Northcott-type finiteness theorem \cite[Theorem 1.6]{Baker2009}
(plus a direct argument to handle the isotrivial case)
we show that the set of parameters $t$
for which this specialization map is injective on $\PrePer(f, K(T))$ is  
cofinite (see Lemma \ref{lem:inj_locus}).
This immediately gives 
\[
    \#{\PrePer}(f_t, K) \ge \#{\PrePer}(f, K(T))
\]
for all but finitely many $t$
    in $K$.
However, it is not \emph{a priori} clear that equality should ever hold.
Our first main result shows it almost always does.

\medskip 
\noindent \textbf{Theorem \ref{thm:little_o}.} 
\textit{Let $f_t \in K(z)$ be a one-parameter family of rational maps 
over a number field $K$.
Let $E \subseteq K$ be the set of parameters $t$ for which the specialization map
\[\PrePer(f, K(T)) \to \PrePer(f_t, K)\] 
is \emph{not} a bijection.
Then the proportion of parameters in $E$ up to height $X$ is vanishingly small as $X \to \infty$, i.e.,
\[
    \Num(E, X) = o(\Num(K, X)). 
\]}

To prove Theorem \ref{thm:little_o}, 
we use Hilbert's irreducibility theorem 
to relate membership in $E$ to the existence of large cycles (see Lemma \ref{lem:surj_locus}).
Then, inspired by the computations of Hutz--Ingram in \cite{HutzIngram}, we use a local argument (the ``$\ell = mrp^e$'' theorem of Zieve \textit{et al}.~\cite[p.~66\textit{ff}]{Silverman}) to show that if $f_t$ has a point of period $\ell$, then $t$ must lie in one of a fixed number of residue classes (Lemma \ref{lem:badred}) modulo a prime on the order of $\sqrt{\ell}$ (Proposition \ref{propn:n=mrpe}).\footnote{A classic result of Call--Silverman \cite[Theorem 4.1]{CallSilverman} relates the canonical heights of $f$ and $f_t$:
\[\hat{h}_{f_t}(P_t) \sim \hat{h}_f(P) h(t) \quad (h(t) \to \infty).\]
This implies that if $P$ is \emph{not} preperiodic, then neither is $P_t$, at least for all but finitely many $t$. Alas, that finite set depends on $P$, suggesting that this line of reasoning cannot be used to prove Theorem \ref{thm:little_o}.}

Subtracting the putative main term $\#{\PrePer}(f, K(T)) \cdot \Num(U(K), X)$ from $\mathcal{A}(X)$
yields the remainder term 
\[
	\mathcal{R}(X) = \!\! \sum_{\substack{t \in E \\ H(t) \le X}} \!\! 
	\big(\#{\PrePer}(f_t, K) - \#{\PrePer}(f, K(T))\big)
\]
which is naturally supported on the exceptional set $E$.
Thus, 
Theorem \ref{thm:little_o} suggests that the average number of preperiodic points 
should equal the generic number of preperiodic points:
\begin{equation} \label{eq:this1}
	\lim_{X \to \infty} \frac{\mathcal{A}(X)}{\Num(K, X)} = \#{\PrePer}(f, K(T)).
\end{equation}
Although we believe it should be possible to prove \eqref{eq:this1} in general,
we were unable to do so without assuming a bit more about the exceptional set $E$.
The issue is that while exceptional portraits may be rare, they may be exceptionally large.
The most we could muster is our second main result.

\medskip
\noindent\textbf{Theorem \ref{thm:avg}.}
\textit{Let $f_t$ be a one-parameter family of rational maps over a number field $K$ 
with domain of definition $U$ and exceptional set $E$, 
and let $m$ be a positive integer.
Then
\begin{align*}
    \sum_{\substack{t \in U(K) \\ H(t) \le X}} \!\!  \#{\PrePer}&(f_t, K)^m \\[-1.2em]
    &= \#{\PrePer}(f, K(T))^m \cdot \Num(U(K), X) + O\big(\Num(E, X) X^{C/\log\log X}\big)
\end{align*}
as $X \to \infty$, where the constants depend on $f$, $n$, and $m$.
If $f$ is a polynomial, the error term may be improved to $O\big(\Num(E, X) (\log X)^m\big)$.}
\smallskip 

The error term in Theorem \ref{thm:avg} comes from estimating the remainder by the $(1,\infty)$-H\"older inequality,
and relies on Troncoso's exponential bound \cite[Corollary 1.3(c)]{Troncoso}:
\begin{equation} \label{eq:i1}
\#{\PrePer}(f_t, K) < 6 \cdot 2^{16d^3 s}
\end{equation}
where $s = s(f_t)$ is the number of bad places of $f_t$ (see Proposition \ref{propn:benetronc}).
Then, using the prime number theorem and several intermediate inequalities (Lemmas \ref{lem:PNT}, \ref{lem:Discheight}, and \ref{lem:Hft_to_Ht}) we show that 
\begin{equation} \label{eq:i2}
s(f_t) \lesssim 2dd'n \log X/\log \log X
\end{equation} 
in the range $H(t) \le X$ as $X \to \infty$, where $d'$ is the degree of $f$ in $t$.
Combining \eqref{eq:i1} and \eqref{eq:i2} and noting that $A^{\log B} = B^{\log A}$ for all $A, B > 0$ yields Theorem \ref{thm:avg}
with any $C > 32d^4 d' nm \log 2$.
We remark that
\[(\log X)^N \ll_N X^{C/\log\log X} \ll_\varepsilon X^\varepsilon\]
for all $N, \varepsilon > 0$.

Numerical evidence suggests that $\Num(E, X)$ grows considerably more slowly than $o(\Num(K, X))$.
This is consistent with the Uniform Boundedness Conjecture, 
which implies that 
\[\Num(E, X) \ll \Num(K, X)^{1/2}\] 
for every one-parameter family
(see Corollary \ref{cor:exceptional_count}).
A classic result of Walde and Russo \cite[Corollary 4]{WaldeRusso}
says that if $f_t(z) = z^2 + t$ has a rational preperiodic point
then the denominator of $t$ is a perfect square. 
As already noted by Sadek \cite[Theorem 4.8]{Sadek2018},
this immediately entails
\[\Num(E, X) \ll \Num(\Q, X)^{3/4}\] 
for this family.
In light of these remarks,
we propose the following new conjecture.

\medskip 
\noindent\textbf{Strong Zero-Density Conjecture.}
\textit{For each one-parameter family of rational maps over a number field $K$, 
there exists a constant $\delta > 0$ 
such that 
\[
    \Num(E, X) 
    \ll \Num(K, X)^{1-\delta}
\]}

Our third main result represents partial progress toward this conjecture.

\medskip
\noindent\textbf{Theorem \ref{thm:den_gives_conj}.}
\textit{Suppose there exists a finite set $S$ of places $v$ of $K$ 
such that $v(t) \ne -1$ for all $v \not \in S$ 
whenever $f_t$ has a nontrivial preperiodic point over $K$.
Then 
\[
	\Num(E, X) \ll \Num(K, X)^{3/4+\varepsilon}.
\]
If $K$ is $\Q$ or an imaginary quadratic field, then we may take $\varepsilon = 0$.}
\smallskip

By generalizing Walde and Russo's argument for quadratic polynomials,
we exhibit (in Section \ref{sec:examples}) several classes of one-parameter families of rational maps satisfying the technical hypothesis in Theorem \ref{thm:den_gives_conj}
(which we call a ``denominator lemma'' for the family).
If a member $f_t$ of any such family has a nontrivial preperiodic point, 
then the denominator ideal of the parameter $t$ 
must be of the form $\mathfrak{a}\mathfrak{b}^2\mathfrak{c}^3$, 
where $\mathfrak{a}$ is one of finitely many ideals (determined by $S$), $\mathfrak{b}$ is arbitrary,
and $\mathfrak{c}$ is squarefree.
Theorem \ref{thm:den_gives_conj} then follows from counting ideals of this form (see Lemma \ref{lem:LangWeberGolomb}).

In particular, combining Theorems \ref{thm:avg} and \ref{thm:den_gives_conj}
gives the first known examples of families of rational maps
over number fields 
satisfying a strong form of the ``Averaged'' Uniform Boundedness Conjecture with an explicit error term (see below).
But this approach is certainly limited: at present, 
we cannot say anything beyond Theorem \ref{thm:little_o} for families \emph{without} a denominator lemma---even one as simple as $f_t(z) = z^2 + t^2$.

The paper concludes with an in-depth discussion of the quadratic family $f_t(z) = z^2 + t$ over $K = \Q$.
Our final result gives new estimates on the remainder term $\mathcal{R}(X)$ in terms of $\Num(\Q, X)^{1/2} \asymp X$.

\medskip
\noindent\textbf{Theorem \ref{thm:HR}.}
\textit{For the quadratic family $f_t(z) = z^2 + t$ over $\Q$, 
\[
C X + O(\sqrt{X} \log X) \le \mathcal{R}(X) \ll X^{3/2} \log \log X \log \log \log X
\]
where
\[C = \frac{12}{\pi^2}\left(\frac{2\pi}{3 \sqrt{3}} + \frac{1 + \sqrt{5}}{2} + 2 \log \frac{1 + \sqrt{5}}{2}\right) \approx 4.607\ldots\]
If no $f_t$ has a cycle of length exceeding 3, then the lower bound is an equality.}
\smallskip

The upper bound in Theorem \ref{thm:HR} follows from results of Walde--Russo \cite[Corollary 4]{WaldeRusso}, Benedetto \cite[Theorem 7.1]{Benedetto}, and Hardy--Ramanujan \cite[Lemma B]{HardyRamanujan},
and improves on the ``trivial'' upper bound $\mathcal{R}(X) \ll X^{3/2} \log X$ obtained by combining Theorems \ref{thm:avg} and \ref{thm:den_gives_conj}.
The lower bound uses Poonen's classification theorem \cite[Figure 1]{Poonen}
for preperiodic point portraits of quadratic polynomials over $\Q$, 
followed by standard techniques from analytic number theory.
The peculiar asymptotic constant,
reflecting the contribution of extra points of period 1 or 2,
arises as the sum of the areas enclosed by 
certain truncated conic sections (see Remark \ref{rmk:geometry}).

As far as we are aware, 
there is only one other result in the literature 
pertaining to the distribution of preperiodic points 
in families of rational functions.
In the recent paper \cite{LeBoudecMavraki},
Le Boudec and Mavraki 
used geometry of numbers techniques 
along with uniform estimates on the canonical height
to show that, 
on average, 
the point at $\infty$ 
is the only rational preperiodic point 
in the multi-parameter family $\Poly_d^* \subset \A^{d-1}$ of 
depressed degree-$d$ polynomials with unit constant term,
\[
\psi(z) = a_d z^d + a_{d-2} z^{d-2} + \ldots + a_1 z + 1 \qquad (a_d \ne 0).
\]
Specifically, they proved \cite[Theorem 1.1]{LeBoudecMavraki} that 
for each $d \ge 2$ and $\varepsilon > 0$,
\[
    \frac{1}{\Num(\Poly_d^*(\Q), X)} 
    \sum_{\substack{\psi \in \Poly_d^*(\Q) \\ H(\psi) \le X}}
    \#{\PrePer}(\psi, \Q) 
    = 1 + O\Big(\frac{1}{X^{\vartheta_d - \varepsilon}}\Big)
\]
where $\vartheta_2 = 1/2$ and $\vartheta_d = 2(d+1)/(5d+1) \in (2/5, 1/2]$ for $d \ge 3$.
Moreover, they conjectured \cite[(1.3)]{LeBoudecMavraki} that the true error term is \[\frac{\gamma_d}{X} + o\Big(\frac{1}{X}\Big)\]
for some constant $\gamma_d > 0$.
Our work overlaps with theirs just in the case $d = 2$, 
where Theorem \ref{thm:HR} improves the error term and suggests a value for $\gamma_2$ (see Remark \ref{rmk:LBM}).

As alluded to above, 
these results may be seen as 
special cases of an ``Averaged'' version of the Uniform Boundedness Conjecture, which we now formulate.
For any algebraic variety $V$,
any field $K$,
and any integer $n$,
write 
\[
V(\bar{K}_n)
= \{P \in V(\bar K) : [K(P) : K] \le n\}
= \bigcup_{[L : K] \le n} V(L) 
\]
for the set of all 
points $P$ of $V(\bar K)$ 
whose field of definition $K(P)$ has degree at most $n$ over $K$.
Let $H : \End^N_d(\Bar \Q) \to \R_{\ge 1}$ be a multiplicative Weil height on the space of degree-$d$ endomorphisms $\phi$ of $\P^N$.

\medskip 
\noindent \textbf{Averaged Uniform Boundedness Conjecture.}
\textit{For each $n, N \ge 1$ and $d \ge 2$, 
\begin{equation} 
\label{eq:avgUBC}
\limsup_{X \to \infty} 
\frac{1}{\Num(\End^N_d(\bar{\Q}_n), X)}
\sum_{\substack{\phi \in \End^N_d(\bar{\Q}_n) \\ H(\phi) \le X}} 
\! \!
\#{\PrePer}(\phi, \Q(\phi)) < \infty.
\end{equation}
That is, the average number of preperiodic points
of degree-$d$ endomorphisms of $\P^N$ defined over number fields of degree at most $n$ is finite.}
\medskip

Certainly, the UBC implies the Averaged UBC, 
as the limit superior is bounded by the Morton--Silverman constant
$C(d, N, n)$. 
Nothing else is known.
Even the number of terms, 
asymptotically equal to $\Num(\P^\nu(\bar\Q_n), X)$ where $\nu = \dim \End^N_d$, is not well understood for general $n$ and $\nu$, 
presumably due to the lack of a field structure on $\bar{\Q}_n$
(see, e.g.,~\cite{GrizzardGunther} and \cite{MasserVaaler} for the state of the art).
The present article addresses the mean value problem \eqref{eq:avgUBC}
with the space $\End^N_d$ replaced by a smaller family $U \subset \P^1$ 
and the set $\bar{\Q}_n$ replaced by a number field $K$.

A related question which the present article does not address but which is nonetheless very interesting concerns the issue of overcounting due to linear conjugacy. 
Let $f : U \to \End^N_d$ be any algebraic family defined over a number field $K$. 
Define an equivalence relation $\sim$ on $U(K)$ by $t \sim t'$ if and only if $f_t$ and $f_{t'}$ are $K$-conjugate,\footnote{i.e.,~there exists $\varphi \in \PGL_{N+1}(K)$ such that $f_t^\varphi := \varphi^{-1} \circ f \circ \varphi = f_{t'}$}
so that 
\[t \sim t' \implies \#{\PrePer}(f_t, K) = \#{\PrePer}(f_{t'}, K).\]
Let $H : U \to \R_{\ge 1}$ be a multiplicative Weil height,
and define
\[
\Tilde{H}([t]) := \inf_{t' \sim t} H(t')
\]
on equivalence classes $[t] \in U(K)/{\sim}$.
Then $\Tilde{H}$ inherits the Northcott property from $H$,
so it makes sense to ask:

\begin{q}
What is the behaviour of
\[\widetilde{\mathcal{A}}(X) := \sum_{\substack{[t] \in U(K)/{\sim} \\ \widetilde{H}([t]) \le X}} \#{\PrePer}(f_t, K)\]
as $X \to \infty$?
\end{q}

We expect $\Tilde{A}(X)$ to be asymptotic to $\Tilde\Num(X) \cdot \#{\PrePer}(f, K(U))$, 
where $\Tilde\Num(X)$ is the number of terms (assuming $U$ is irreducible).
Note that if the family has any symmetries---automorphisms $\sigma$ of $U$ such that $f_{\sigma(t)} = f_t$ for all $t$ in $\Bar K$---then $t \sim \sigma(t)$ trivially. 
If this be the only source of overcounting, 
then Hilbert's irreducibility theorem suggests
$\Tilde{\mathcal{A}}(X) \sim \frac{1}{g} \cdot \mathcal{A}(X)$ where $g$ is the number of symmetries $\sigma$ defined over $K$.
Certainly when $U \subset \P^1$ and $N = 1$ the analogue of Theorem \ref{thm:avg} holds, by the same proof.

Our belief is that nontrivial overcounting occurs just when the induced map to moduli space $U \to \mathsf{M}^N_d$ is not finite.
Of course, this map is constant if the family is isotrivial 
(the converse holds by Petsche--Szpiro--Tepper \cite[Theorem 1]{PST} assuming $\dim U = 1$).
In this case, $\#{\PrePer}(f_t, K)$ is uniformly bounded (see Remark \ref{rmk:iso}); and so ``average value'' = ``generic value'' 
provided that $\Tilde\Num(X)$ does not grow too slowly compared to $\Num(E/{\sim}, X)$ (which should satisfy the analogue of Theorem \ref{thm:little_o}).

For instance, if $f_t(z) = z + t/z$ where $U = \A^1 \setminus \{0\}$
then $t \sim t'$ if and only if $t/t' \in (K^\times)^2$
\cite[Example 4.69]{Silverman}.
When $K = \Q$, the possible preperiodic point portraits of this (isotrivial) family have been completely classified by Levy--Manes--Thompson \cite[Corollary 4.4]{LevyManesThompson};
from their work, it follows that 
\[\Tilde{\mathcal{A}}(X) = 2 \cdot \Tilde\Num(X) + 6 \qquad (X \ge 2).\]
As pointed out to the author by Koukoulopoulos, $\Tilde\Num(X)$ is twice the number of squarefree integers which may be written as a product of two positive integers $\le X$; thus, by Ford's solution of the multiplication table problem \cite{Ford}, 
\[
\widetilde{\mathcal{N}}(X) 
\asymp \frac{X^2}{(\log X)^\delta (\log \log X)^{3/2}}
\]
where $\delta = 1 - (1 + \log \log 2)/\log 2$.
It seems quite challenging to estimate $\Tilde\Num(X)$ in general.

The paper is laid out as follows.
Section \ref{sec:background} sets down notational conventions and recalls standard results.
Section \ref{sec:spec} carefully defines one-parameter families
and proves several fundamental Lemmas about the objects associated to them (\textit{viz.}~the domain of definition, the resultant polynomial, the injectivity and surjectivity loci, the exceptional set, and the sets $Z_l$).
These are used in Section \ref{sec:little_o} to prove that almost all specializations are portrait-preserving (Theorem \ref{thm:little_o}) 
as well as to discuss potential improvements of said result (Proposition \ref{propn:equivalences1}).
Section \ref{sec:avg}, concerning moments of portrait size, contains the proof of Theorem \ref{thm:avg} and gives some criteria for ``average value = generic value'' (Proposition \ref{propn:equivalences2}).
Section \ref{sec:denominator} proves that families with a denominator lemma satisfy the Strong Zero-Density Conjecture (Theorem \ref{thm:den_gives_conj}),
and Section \ref{sec:examples} exhibits three classes of Examples of such families using the method of ``tropicalization''.
Section \ref{sec:quadratic} is devoted to the quadratic family, culminating in Theorem \ref{thm:gold}.

\section*{Acknowledgements}

This paper constitutes part of my Ph.D.~thesis \cite[Chapter 2]{thesis}. 
I would like to thank my advisor, Patrick Ingram, for his helpful comments throughout the research and writing process.
I am also grateful to 
Matt Baker and Rob Benedetto for answering my questions about isotriviality and good reduction,
Vesselin Dimitrov for teaching me about thin sets, 
Sacha Mangerel for computing the mean value of $\omega(n) \log \omega(n)$, and 
Shuyang Shen for assisting me with numerical experiments.
The generous feedback from the anonymous referees greatly improved the quality of this article; I am particularly indebted to the reviewer who spotted a subtle gap in Lemma \ref{lem:inj_locus}, and to Jonathan Love who helped me fix it.
Finally, I wish to thank the University of Toronto for its financial support during my final year.

\section{Background} \label{sec:background} 

Throughout this paper, 
$L$ denotes an arbitrary field 
(of arbitrary characteristic)
and $K$ denotes a number field (i.e.,~a finite extension of $\Q$).
Once and for all we let
\[
\boxed{n = [K : \Q]}
\]
be the degree of $K$ over $\Q$.

\subsection{Asymptotics}
Let $f, g : (a, \infty) \to \R$ with $g$ eventually positive, 
and let 
\[C = \limsup_{x \to \infty} \frac{f(x)}{g(x)}.\]
We write
$f = O(g)$ (or $f \ll g$) iff $C < \infty$, 
$f \lesssim g$ iff $C \le 1$, and $f = o(g)$ iff $C = 0$.
If $f$ is also eventually positive, 
we write $f \gg g$ (resp.~$f \gtrsim g$) iff $g \ll f$ (resp.~$g \lesssim f$),
and $f \asymp g$ (resp.~$f \sim g$) iff $f \ll g \ll f$ (resp.~$f \lesssim g \lesssim f$).

By a completely standard abuse of notation, 
we occasionally use $O(g)$ (resp.~$o(g)$) 
to denote an unspecified function which is big-$O$ (resp.~little-$o$) of $g$.

A convenient definition is the following.
Let $h : (a, \infty) \to (0, \infty)$ be nondecreasing.
We say $h$ is \emph{nice} 
	(resp.~\emph{very nice}) 
if $h(cx) \ll h(x)$ 
	(resp.~$h(cx) \lesssim c h(x)$) 
for each $c > 0$.
If $f(x) \to \infty$ and $f \ll g$, 
then $h \circ f \ll h \circ g$ for any nice $h$. 
The same holds for very nice $h$ with $\ll$ replaced by $\lesssim$.
For example, 
if $\alpha > 0$ then $x^\alpha$ is nice; 
and the function $x \log x$ is very nice.

\subsection{Valuations}
A \emph{valuation} on a field $L$ is a surjective function $v : L \twoheadrightarrow \Z \cup \{\infty\}$
such that
\begin{enumerate}
\item	$v(x) = \infty$ iff $x = 0$
\item	$v(xy) = v(x) + v(y)$
\item	$v(x + y) \ge \min\{v(x), v(y)\}$
\end{enumerate}
for all $x, y \in L$. 
(We only consider normalized discrete valuations.)
Every valuation on a number field $K$ 
is a \emph{$\mathfrak{p}$-adic valuation}
\[
	\ord_{\mathfrak{p}}(x) = \text{exponent of $\mathfrak{p}$ in the factorization of $xO_K$}
\]
for some prime ideal $\mathfrak{p}$. 
Similarly, 
every valuation on a function field $L(T)$ that is trivial on $L$ 
is either the $\pi$-adic valuation $\ord_\pi$ for some irreducible polynomial $\pi$, 
or else the \emph{degree} valuation $\ord_\infty$, given by $-{\deg}$ on polynomials.

\subsection{Absolute values}

An \emph{absolute value} on $L$ is a function $| \cdot | : L \to \R_{\ge 0}$ such that 
\begin{enumerate}
\item	$|x| = 0$ iff $x = 0$
\item	$|xy| = |x| \cdot |y|$
\item	$|x + y| \le |x| + |y|$
\end{enumerate}
for all $x, y \in L$.
If (3) can be replaced by 
\begin{itemize}
\item[(3')] $|x + y| \le \max\{|x|, |y|\}$ 
\end{itemize}
then we say $| \cdot |$ is \emph{nonarchimedean}.
Two absolute values $| \cdot |$ and $| \cdot |'$ are \emph{equivalent} 
iff $| \cdot |' = | \cdot |^\alpha$ for some positive real number $\alpha$;
an equivalence class is called a \emph{place}.
Every valuation $v$ on $L$ defines a nonarchimedean place of $L$ 
by setting $|x|_v = q_v^{-v(x)}$ for any fixed $q_v > 1$.

The \emph{standard absolute values} on $\Q$ are 
the usual absolute value $|x|_\infty = |x|$
and the $p$-adic absolute values $|x|_p = p^{-v_p(x)}$.
For a number field $K$, 
we let $M_K$ denote the set of absolute values of $K$ extending any standard absolute value on $\Q$.
By Ostrowski's theorem, 
every $v$ in $M_K$ is either the archimedean absolute value 
\[
	\lvert x \rvert_\sigma = \lvert \sigma(x) \rvert 
\]
for some embedding $\sigma : K \into \C$, with local degree 
\[
n_v = [K_v : \R] = \begin{dcases*} 1 & if $\sigma$ is real, \\ 2 & if $\sigma$ is complex; \end{dcases*}
\]
or else the nonarchimedean absolute value 
\[
	\lvert x \rvert_\mathfrak{p} 
 = p^{-{\ord_\mathfrak{p}(x)}/e_\mathfrak{p}}
\]
for some prime ideal $\mathfrak{p} \mid p$, where $e_\mathfrak{p} = \ord_\mathfrak{p}(p)$ is the ramification index; the local degree is
\[
n_v = [K_v : \Q_p] = e_\mathfrak{p} f_\mathfrak{p}
\]
where $f_\mathfrak{p} = \dim_{\F_p} O_K/\mathfrak{p}$ is the inertia degree.

\subsection{Height} \label{sec:def_heights}
Let $|\cdot|$ be an absolute value on $L$.
We extend $|\cdot|$ to lists of elements of $L$ 
by setting 
\[|a_1, \ldots, a_k| = \max\{|a_1|, \ldots, |a_k|\}\]
and similarly to lists of polynomials over $L$
(viewed as lists of coefficients).

The \emph{(absolute multiplicative) height} 
of a point $P = (a : b) \in \P^1(K)$ 
defined over a number field $K$ of degree $n$
is the product
\[H(P) = \prod_{v \in M_K} |a, b|_v^{n_v/n} \in [1, \infty)\]
over all standard absolute values of $K$.
For any subset $A \subseteq \P^1(K)$ 
and any real number $X \in \R$
the \emph{counting function} of $A$ is 
\[
    \Num(A, X) = \#\{P \in A : H(P) \le X\}.
\]
For us, 
the most important counting result is Schanuel's theorem, 
which says that
\begin{equation}\label{eq:schanuel}
    \Num(\P^1(K), X) = c_K X^{2n} + O\big(X^{2n-1} \log X\big)
\end{equation}
where $c_K$ is a constant involving all the classical invariants of $K$, 
and the $\log X$ disappears if $n > 1$ \cite[Corollary]{Schanuel}.
See the paper \cite{GrizzardGunther} for many other interesting counting results.

\subsection{Ideal norm}
Let $\mathcal{I}_K$ denote the set of nonzero integral ideals of $O_K$.
The \emph{(absolute) norm} of an ideal $I \in \mathcal{I}_K$ 
is the positive integer
\[
    \Nm I = [O_K : I] = \#(O_K/I).
\]
For each $\mathcal{B} \subseteq \mathcal{I}_K$ 
the \emph{counting function} of $\mathcal{B}$ 
is
\[
    \Num(\mathcal{B}, X) = \#\{I \in \mathcal{B} : \Nm I \le X\}.
\]
We will need to know the number of integral ideals up to a given norm $X$.
This is provided by Weber's theorem, according to which
\begin{equation} \label{eq:weber}
	\Num(\mathcal{I}_K, X) = \rho_K X + O(X^{1-1/n})
\end{equation}
where $\rho_K$ is another constant depending on $K$ (see, e.g.,~\cite[Theorem 5]{MurtyVanOrder}).

\subsection{Thin sets}

A subset $\Omega$ of $\A^1(K)$ is called \emph{thin} 
in the sense of Serre \cite[Section 9.1]{Serre} 
if there exists an algebraic variety $X$ over $K$ and a morphism $\pi : X \to \A^1$ such that 
\begin{enumerate}[(a)]
\item $\Omega \subset \pi(X(K))$, 
\item the generic fibre of $\pi$ is finite, and
\item $\pi$ has no rational section over $K$.
\end{enumerate}
It is immediate that the class of thin sets 
contains $\varnothing$
and is closed under 
    taking finite unions 
    and passing to arbitrary subsets---i.e.,~it is an \emph{ideal of sets}.

We may now state 
a version of Hilbert's irreducibility theorem 
most convenient for us:
If $\Phi \in K[T][X, Y]$ 
is homogeneous in $X, Y$ 
and has no roots over $K(T)$, 
then the parameters $t \in K$ 
for which $\Phi_t$ acquires a root over $K$
constitute a thin set.
To see this,
let $X = \{(t, P) \in \A^1 \times \P^1 : \Phi_t(P) = 0\}$
and let $\pi : X \to \A^1$ be the projection.
Then 
\begin{enumerate}[(a)]
    \item $\pi(X(K)) 
    = \{t \in \A^1(K) : \Phi_t(P) = 0 \text{ for some } P \in \P^1(K)\}$,
    \item generic fibre of $\pi$ has cardinality $\deg \Phi$, and
    \item the data of a rational section $\sigma : \A^1 \dashrightarrow X$ of $\pi$ defined over $K$ 
    consists of an open set $U \subseteq \A^1$ and a pair $a, b \in K[T]$ such that $\Phi_t(a(t), b(t)) = 0$ for all $t \in U(K)$; 
    but $\Phi$ has no roots over $K(T)$, 
    so $\pi$ has no rational section over $K$.
\end{enumerate}

The key feature of thin sets is that their counting functions are asymptotically negligible. 
Specifically, 
if $\Omega \subseteq \A^1(K)$ is thin, 
then we have
\begin{equation} \label{eq:thin}
    \Num(\Omega, X) \ll X^n
\end{equation}
where $n$ is the degree of the number field $K$
(see \cite[Section 9.7]{Serre}).

\subsection{Dynamics} \label{ss:dynamics}
Let $f : X \to X$ be a self-map of a set $X$. 
A point $x \in X$ is called 
\begin{itemize}
    \item \emph{periodic}
if $f^l(x) = x$ for some $l \ge 1$;
    \item \emph{preperiodic} 
if $f^t(x)$ is periodic for some $t \ge 0$.
\end{itemize}
The least $l$ (resp.~$t$) satisfying the above is called the \emph{period} (resp.~\emph{tail length}) of $x$.
The \emph{type} of a preperiodic point is the symbol $l_t$, where $t$ is the tail length of $x$ and $l$ is the period of $f^t(x)$.
A \emph{fixed point} is a point of period 1.
The sets of fixed, periodic, and preperiodic points of $f$
are denoted $\Fix(f)$, $\Per(f)$, and $\PrePer(f)$ respectively.
The \emph{backward orbit} 
of a subset $A \subseteq X$ 
is the set 
\[
    f^{-\infty}(A) = \{x \in X : f^k(x) \in A \text{ for some } k \ge 0\}.
\]

The dynamics of a self-map may be visualized 
in terms of the associated \emph{portrait},
defined as the functional digraph with vertex-set $X$ 
and an arrow from $x$ to $x'$ if and only if $x' = f(x)$.
Terms like ``cycle'' and ``branch'' refer to this visualization.

A \emph{homomorphism} of self-maps $f : X \to X$ to $g : Y \to Y$ 
is a homomorphism of the underlying graphs, 
i.e.,~a function $\varphi : X \to Y$ such that $\varphi(f(x)) = g(\varphi(x))$ for all $x \in X$. 
Homomorphisms send (pre)periodic points to (pre)periodic points; 
in fact, if $x$ has type $l_t$ then $\varphi(x)$ has type $l'_{t'}$ where $l' \mid l$ and $t' \le t$.

\subsection{Rational maps}

A \emph{rational function} with coefficients in $L$ is an element $\phi$ of $L(z)$,
say
\[\phi(z) = \frac{f(z)}{g(z)}\] 
with $f$ and $g$ coprime polynomials.
The \emph{degree} of $\phi$,
denoted $\deg \phi$, 
is the maximum of the degrees of $f$ and $g$.
If $\deg \phi = d$ then by homogenizing $f$ and $g$ to a pair of degree-$d$ 
binary forms $F$ and $G$ 
we obtain a rational map $\P^1 \dashrightarrow \P^1$, 
also denoted $\phi$, 
such that 
\[\phi(a : b) = (F(a, b) : G(a, b))\]
for all $(a : b) \in \P^1$.
Any such pair $F, G$ is called a \emph{lift} of $\phi$; 
any two lifts are congruent modulo $L^\times$.
Since $f$ and $g$ are coprime, 
every rational map obtained in this way 
is automatically a morphism
(i.e.,~defined everywhere).
Conversely, every endomorphism of $\P^1$ 
dehomogenizes to an element of $L(z) \cup \{\infty\}$.

A rational map $\phi$ is constant if and only if $\deg \phi = 0$.
Rational maps of degree $1$ are precisely the automorphisms of $\P^1$,
also known as M\"obius maps:
\[\phi(z) = \frac{az + b}{cz + d} \qquad (ad - bc \ne 0).\]

The set of $L$-rational preperiodic points of $\phi$ (including, possibly, $\infty$) is written
\[
    \PrePer(\phi, L) = \PrePer(\phi) \cap \P^1(L).
\]

The projectivization of the vector space over $L$ 
comprising all pairs of degree-$d$ binary forms with coefficients in $L$
is denoted $\bar{\Rat}_d(L)$.  
Its dimension is $2d + 1$.
The space 
\[\Rat_d \subseteq \bar{\Rat}_d\] 
of rational maps of degree (exactly) $d$
is the complement of the resultant locus,
which we define next.

\subsection{Resultants}
Two degree-$d$ binary forms 
\begin{align*}
F(X, Y) &= F_d X^d + \ldots + F_0 Y^d, \ G(X, Y) = G_d X^d + \ldots + G_0 Y^d \in L[X, Y]
\end{align*}
have no common zeroes over $\bar{L}$ 
if and only if 
their \emph{resultant} 
\[
	\Res(F, G) = \det \begin{pmatrix}
	F_d     &        &         & G_d     &        &         \\
	\vdots  & \ddots &         & \vdots  & \ddots &      \\
	F_0     &        & F_d     & G_0     &        & G_d \\
	        & \ddots & \vdots  &         & \ddots & \vdots  \\
	        &        & F_0     &         &        & G_0     
\end{pmatrix} \in L
\]
is nonzero. 
The $(2d)$-by-$(2d)$ matrix on the right 
will be called the \emph{Sylvester matrix}. 
Beware that this is not the same as the Sylvester matrix of the dehomogenizations of $F$ and $G$ unless those have equal degree.

The resultant is homogeneous of degree $d$ in each argument 
and commutes with ring homomorphisms out of $L$. 
It also satisfies the following \emph{composition law} \cite[Exercise 2.12(a)]{Silverman}: 
if $\deg F = \deg G = D$ and $\deg f = \deg g = d$
then 
\begin{equation} \label{eq:res_comp}
\Res(F(f, g), G(f, g)) = \Res(F, G)^d \Res(f, g)^{D^2}.
\end{equation}

\subsection{Good reduction}

Let $v$ be a valuation on $L$ and let $F, G \in L[X, Y]$ be binary forms of degree $d$.
For all $c$ in $L$, 
\[v(\Res(cF, cG)) = 2d \, v(c) + v(\Res(F, G))
\quad\text{and}\quad 
v(cF, cG) = v(c) + v(F, G).\]
Thus, the function
\[(F, G) \mapsto v(\Res(F, G)) - 2d \, v(F, G)\]
descends to a well-defined non-negative map on $\bar{\Rat}_d(L)$, 
denoted $v(\Res\phi)$, 
which is finite if and only if $\phi \in \Rat_d(L)$.
We say $\phi$ has \emph{good reduction} at $v$ 
if and only if $v(\Res\phi) = 0$.

An equivalent definition of good reduction is the following.
Let $\ell_v = O_v/\mathfrak{m}_v$ be the residue field of $L$,
and let $F, G$ be a \emph{$v$-normalized} lift of $\phi$, 
meaning $v(F, G) = 0$.
The \emph{reduction} of $\phi$ at $v$ 
is the element $\phi_v$ of $\bar{\Rat}_d(\ell_v)$
defined by the lift
\[
  \tilde{F}(X, Y) = \tilde{F}_d X^d + \ldots + \tilde{F}_0 Y^d, 
\ \tilde{G}(X, Y) = \tilde{G}_d X^d + \ldots + \tilde{G}_0 Y^d \in \ell_v[X, Y]
\]
obtained by ``reducing coefficients modulo $v$''.
Now $\phi$ has good reduction at $v$ 
if and only if $\phi_v \in \Rat_d(\ell_v)$. 
The equivalence of these two definitions is standard; see \cite[Theorem 2.15 and Proposition 4.95(a)]{Silverman}.

Over a number field $K$,
the valuations of the resultant of a rational map $\phi \in \Rat_d(K)$
may be packaged into a single object,
the \emph{resultant ideal}: 
\[\Res \phi = \prod_{\mathfrak{p}} \mathfrak{p}^{\ord_{\mathfrak{p}}({\Res \phi})}.\]

\section{Specialization} \label{sec:spec}

\subsection{Domain of definition} \label{sec:domain_of_defn}

Informally, the specialization at a parameter $t$ in $L$ 
of an object defined over $L(T)$ 
is obtained by plugging in $t$ for $T$. 
One way to make this precise is to recognize specialization at $t$ 
as reduction modulo the place $v = \ord_{T - t}$ of $L(T)$.
Since the residue field of $v$ is isomorphic to $L[T]/(T - t) = L$, 
this entails the following:
\begin{itemize}
\item If $P \in \P^1(L(T))$, then $P_t = (a(t) : b(t))$ where $a, b \in L[T]$ satisfy $(a : b) = P$, are both defined at $t$, and do not both vanish at $t$.
\item If $f \in \Rat_d(L(T))$ is a rational map, 
then $f_t \in \bar{\Rat}_d(L)$ is the rational map given by $f_t = (F_t : G_t)$ where $F, G \in L(T)[X, Y]$ is any lift of $f$ such that $t$ is not a pole of any coefficient of $F$ nor of $G$ and at least one coefficient of $F$ or of $G$ is nonzero at $t$ (i.e.,~any $v$-normalized lift).
\end{itemize}
Thus, 
a rational map $f$ defined over $L(T)$ 
yields a one-parameter family of rational maps $f_t$ defined over $L$ 
(though not necessarily all of the same degree as $f$).

\begin{defn}
Let $f \in L(T)(z)$ be a rational map.
The \emph{domain of definition} (over $L$) of the family $f_t$ is the set 
\[
	U(L) = \{ t \in \A^1(L) : \deg f = \deg f_t \}.
\]
\end{defn}

Since $L[T]$ is a principal ideal domain, 
each point $P$ and each rational map $f$ 
admits a representative that is ``globally defined'' with respect to specialization---%
first by clearing denominators, 
then by cancelling common factors.
For points, this results in a pair of coprime polynomials $a, b \in L[T]$.
For rational maps, this results in a pair of polynomials $F, G \in L[T][X, Y]$
whose coefficients are altogether coprime. 
In either case, 
we say that the pair $a, b$ (resp.~$F, G$) is an \emph{integral lift} of $P$ (resp.~$f$).

\begin{propn} \label{propn:specialization}
Let $F, G \in L[T][X, Y]$ be an integral lift of $f$.
For each $t$ in $L$, 
\[\deg f = \deg f_t 
\ \text{ if and only if } \ 
\Res(F, G)(t) \ne 0.\]
In particular, the domain of definition is a Zariski-open subset of $\A^1$: 
\[U = V(\Res(F, G))^c.\]
\end{propn}

\begin{proof}
This follows immediately from the two equivalent definitions of good reduction applied to the place $v = \ord_{T - t}$ of $L(T)$, 
plus the observation that an integral lift is $v$-normalized at every place $v \ne \infty$ of $L(T)$.
\end{proof}

By analogy with the resultant ideal, 
and in light of Proposition \ref{propn:specialization},
we make the following definition.

\begin{defn}
The \emph{resultant polynomial} of the one-parameter family $f_t$ 
is the resultant $R_f := \Res(F, G) \in L[T]$ of any integral lift of $f$.
\end{defn}

Strictly speaking, $R_f$ is only defined modulo $(L^\times)^{2d}$ 
but this will never be an issue for us.

\begin{propn} \label{propn:deg_Rf}
Let $F, G \in L[T][X, Y]$ be an integral lift of $f$.
Let $d$ be the common degree (in $X$ and $Y\!$) of $F$ and $G$ [the \emph{degree} of $f$]
and let $d'$ be the maximum of the degrees (in $T\!$) of the coefficients of $F$ and $G$ [the \emph{co-degree} of $f$].
Then 
\[ \deg R_f \le 2dd'. \] 
Equality holds if and only if the specialization of $f$ at $t = \infty$ has degree $d$.
\end{propn}

\begin{proof}
Let $v = \ord_\infty$ be the place at $\infty$.
Then $v(T^{-d'} F, T^{-d'} G) = 0$, 
so 
\[ v\big({\Res}(T^{-d'} F, T^{-d'} G)\big) = v(\Res f) \ge 0 \] 
with equality if and only if $\deg f_\infty = d$.
By homogeneity of the resultant, 
\[ v\big({\Res}(T^{-d'} F, T^{-d'} G)\big) = v\big(T^{-2dd'} {\Res}(F, G)\big) = 2dd' - \deg R_f. \qedhere\]
\end{proof}

Proposition \ref{propn:deg_Rf} suggests that the domain of definition ought to be a subset of $\P^1$ instead of $\A^1$, and even gives a criterion for membership of $\infty$ therein depending solely on the resultant polynomial of the family; but as a matter of convenience we shall not take this stance.
Thus, $\infty$ is always excluded.

\begin{rmk} \label{rmk:codegree}
Any one-parameter family $f_t$
induces a $U$-morphism $U \times \P^1 \to U \times \P^1$,
namely $(t, P) \mapsto \big(t, f_t(P)\big)$.
It also induces a rational map 
$\A^1 \dashrightarrow \Rat_d$, 
given by $t \mapsto f_t$,
whose domain and degree 
are precisely 
the domain of definition $U$ and the co-degree $d'$.
These viewpoints are in fact equivalent; see \cite[Proposition 4.31]{Silverman}.
\end{rmk}

\begin{rmk} \label{rmk:composition}
Let $f_t, g_t$ be one-parameter families of rational maps over $L$ with domains of definition $U$ and $V$, respectively. 
Then by \cite[Theorem 2.18(b)]{Silverman}, 
the domain of definition of the composite family $f_t \circ g_t$ contains $U \cap V$, though it may be bigger.\footnote{For example, $f(z) = T^a z^d$ and $g(z) = z^e/T^b$ have domain of definition $\A^1 \setminus \{0\}$ assuming $a, b > 0$, yet $f(g(z)) = T^{a-bd} z^{de}$ has domain of definition $\A^1$ assuming $a = bd$.
In general, the composition of $v$-normalized lifts need not be $v$-normalized.}
The resultant polynomial 
of $f \circ g$ is given by the formula 
\begin{equation} \label{eq:resFoG}
R_{f \circ g}(T) = \frac{R_f(T)^e R_g(T)^{d^2}}{h(T)^{d + e}}
\end{equation}
where $h \in L[T]$ is the g.c.d.~of the coefficients of the composition $F_0(G_0, G_1)$, $F_1(G_0, G_1)$ of any two integral lifts $F_0, F_1$ and $G_0, G_1$ of $f$ and $g$.
Formula \eqref{eq:resFoG} may be derived from 
\eqref{eq:res_comp} and homogeneity of the resultant.
\end{rmk}

\begin{rmk} \label{rmk:iteration}
By a result of Benedetto \cite[Theorem B]{Benedetto2001},
which says that good reduction is a dynamical property,
$f$ and every iterate of $f$ have the same domain of definition.
\end{rmk}

\subsection{Special loci} 

Specialization at $t \in U$ induces a homomorphism of self-maps from $f$ to $f_t$,
in the sense that 
\[f(P)_t = f_t(P_t)\]
for all $P$ in $\P^1(L(T))$ \cite[Theorem 2.18(a)]{Silverman}.
By iteration, $(f^n(P))_t = f_t^n(P_t)$ for each $n \ge 0$.
In particular, if $P$ is preperiodic under $f$, 
then $P_t$ is preperiodic under $f_t$.
For brevity, write
\[
    \Gamma = \PrePer(f, L(T)) \quad\text{and}\quad \Gamma_t = \PrePer(f_t, L).
\]

\begin{defn}
The \emph{injectivity} (resp.~\emph{surjectivity}, \emph{bijectivity}) \emph{locus} of $f$ is the set of parameters $t$ in $U$ for which the specialization map $\Gamma \to \Gamma_t$ is injective (resp.~surjective, bijective).
The complement in $U$ of the bijectivity locus is called the \emph{exceptional set} and is denoted $E$.  
\end{defn}

Note that our notion of exceptional set is unrelated to the ``exceptional sets'' of complex dynamics (comprising points with finite grand orbit).

\begin{lem} \label{lem:inj_locus}
Let $f$ be a one-parameter family of rational maps over $L$ with domain of definition $U$.
Then the injectivity locus
\[
    U^\circ(L) = \{t \in U(L) : \text{the specialization map } \Gamma \to \Gamma_t \text{ is injective}\}
\]
is a cofinite subset of $U(L)$.
\end{lem}

\begin{proof}
First, assume $f$ is isotrivial, 
i.e.,~there exists a M\"obius map $\psi$ defined over 
a finite extension $M$ of $L(T)$
such that $g := \psi \circ f \circ \psi^{-1}$ 
is defined over the field of constants of $M$.
Given a valuation $w$ of $M$ and a point $t$ in $L$,
write $w \mid t$ to mean $w$ lies over the valuation $v = \ord_{T - t}$ of $L(T)$ (i.e.,~$\mathfrak{m}_w \cap O_v = \mathfrak{m}_v$).
We claim that 
\begin{equation} \label{eq:dom_def_psi}
\begin{gathered}
\{w : w \mid t \text{ for some $t$ in $\A^1(L)$ and $\psi$ has good reduction at $w$}\} 
\\
=
\{w : w \mid t \text{ for some $t$ in $U(L)$}\}.
\end{gathered}
\end{equation}
For starters, suppose $w \mid t$.
Then $g_w = g$ 
because $w$ is trivial on the field of constants of $M$,
and $f_w = f_t$ because any integral lift of $f$ 
is $w$-normalized (as $w = e \cdot v$ on $L(T)$ for some $e \ge 1$).
In particular, 
$g$ always has good reduction at $w$; 
and $f$ has good reduction at $w$ if and only if $t \in U(L)$.

The inclusion $\subseteq$ in equation \eqref{eq:dom_def_psi} is now immediate: if $\psi$ has good reduction at $w$, 
then by direct calculation of the resultant so does $\psi^{-1}$, and hence by \cite[Theorem 2.18(b)]{Silverman} 
so does the composition $\psi^{-1} \circ g \circ \psi = f$.
For the other inclusion, 
let $w \mid t$ for some $t \in U(L)$,
and let $\zeta \in \P^1_\an$ denote the Gauss point in the 
Berkovich projective line over the completion $M_w$.
Since $g$ has (explicit) good reduction at $w$, 
we have $g^{-1}(\{\zeta\}) = \{\zeta\}$;
thus
\[
    f^{-1}\big(\{\psi^{-1}(\zeta)\}\big) 
    = \psi^{-1}\big(g^{-1}(\{\zeta\})\big)
    = \{\psi^{-1}(\zeta)\}
\]
    which says $\psi^{-1}(\zeta)$ is a completely invariant hyperbolic point of $f$.
    Since $\deg f > 1$ such a point is unique,
    and since $\deg f_w = \deg f$ this point must be $\zeta$.
    It follows that $\psi(\zeta) = \zeta$, i.e.,~$\psi$ has nonconstant---hence good---reduction at $w$.\footnote{This neat argument was shown to the author by Rob Benedetto.
See 
    \cite[Lemma 2.17 and Propositions 2.15, 10.5, 10.45]{BakerRumely_book}
or
    \cite[Corollary 7.16, Theorem 7.34, Propositions 8.12 and 8.13]{Benedetto_book}
for the relevant statements from the theory of Berkovich space.}

Since $\psi$ is a homomorphism from $f$ to $g$ (cf.~\S \ref{ss:dynamics})
we get $\psi(\Gamma) \subseteq \PrePer(g, M)$.
But because $\deg g = \deg f > 1$, 
the preperiodic points of $g$ 
are algebraic over any field of definition of $g$.
Thus, for each $P \in \Gamma$, 
$\psi(P)$ is defined over the field of constants of $M$.

At last, let $t \in U(L)$.
By Chevalley's extension theorem, 
there exists a valuation $w$ on $M$ lying over $t$.
By \eqref{eq:dom_def_psi}, 
$\psi$ has good reduction at $w$.
If $P, Q \in \Gamma$ with $P_t = Q_t$
then by the preceding paragraph
\begin{align*}
\psi(P) = \psi(P)_w = \psi_w(P_w) &= \psi_w(P_t) \\
                    &= \psi_w(Q_t) = \psi_w(Q_w) = \psi(Q)_w = \psi(Q).
\end{align*}
Since $\psi$ is invertible, $P = Q$. 
Thus $t \in U^\circ(L)$,
which shows that $U^\circ = U$.

Second, assume $f$ is \emph{not} isotrivial.
Then by Baker's theorem \cite[Theorem 1.6]{Baker2009}, 
$\Gamma$ is finite.
Writing $\Gamma = \{P_1, \ldots, P_m\}$ where $P_i = (a_i : b_i)$ with $a_i, b_i  \in L[T]$ coprime,
we see that $(P_i)_t = (P_j)_t$ for some $i \ne j$ if and only if 
$t$ is a root of the homogeneous Vandermonde determinant 
\[D_\Gamma 
= \prod_{i < j} P_j \wedge P_i 
= \prod_{i < j} (a_j b_i - a_i b_j) 
\in L[T].\]
Thus, $U^\circ = U \setminus V(D_\Gamma)$.
\end{proof}

\begin{rmk}
If $f$ is isotrivial, then $\Gamma$ may be infinite. 
However, if $L$ is a number field, 
then by Lemma \ref{lem:inj_locus} and Northcott's theorem there exists a finite set into which $\Gamma$ embeds.
Thus, over number fields, 
the generic number of preperiodic points is always finite.
\end{rmk}

The surjectivity locus is a bit more subtle, 
and is best understood by considering its complement.
Essentially, surjectivity fails if and only if $\Gamma$ 
acquires a new branch or a new cycle upon specialization.
New branches arise when $f_t^{-1}(P_t) \supsetneq f^{-1}(P)_t$ for some $P$ in $\Gamma$, i.e.,~%
when $P_t$ has more $L$-rational preimages by $f_t$ 
than just the specializations of the $L(T)$-rational preimages of $P$ by $f$;
thus, new branches are governed by the irreducibility of $f(z) - P$.
By contrast, new cycles arise as unexpected roots of the period polynomials $f^l(z) - z$, 
and may be arbitrarily long.
These remarks serve as the basis for our next result. 
Henceforth, we work over a number field $K$.
Write $(a, b) \wedge (c, d)$ to mean $ad - bc$.

\begin{lem} \label{lem:surj_locus}
Let $f_t$ be a one-parameter family of rational maps over $K$ with domain of definition $U$, 
and let 
\[\Sigma = \{t \in U(K) : \text{the specialization map } \Gamma \to \Gamma_t \text{ is surjective}\}\] 
be the surjectivity locus.
For each $l \ge 0$ define 
\[Z_l = \{t \in U(K) : f_t \text{ has a $K$-rational cycle of length exceeding } l\}.\]
Let $l_0$ be the maximum period among points of $\Gamma$ (set $l_0 = 0$ if $\Gamma = \varnothing$).
Then for all $l \ge l_0$, 
the symmetric difference\footnote{We use the symmetric difference here for counting purposes later.} 
\[
    \Sigma^c \mathbin{\triangle} Z_l = (\Sigma \cap Z_l) \cup (\Sigma^c \cap Z_l^c)
\] 
is a thin subset of $K$ (which depends on $f$ and $l$).
\end{lem}

\begin{proof}
The length of a cycle cannot increase under specialization, 
so $\Sigma \cap Z_l = \varnothing$ for all $l \ge l_0$.
Thus, it suffices to show that each $\Sigma^c \cap Z_l^c$ is thin.
This is clear for $l < l_0$, as $Z_l \supseteq U^\circ(K)$ 
in this range.
So, 
let $l \ge l_0$ and let $m_0$ be the maximum tail length among points of $\Gamma$ (set $m_0 = -1$ if $\Gamma$ is empty).
Using an integral lift of $f$, put 
\[
    \Phi = f^{l! + m_0 + 1} \wedge f^{m_0 + 1} \in K[T][X, Y].
\]
Note that $\Phi$ is nonzero because $l! > 0$ and $\deg f > 1$.
Factor $\Phi = \Phi_0 \Phi_1$ where $\Phi_0$ splits completely\footnote{Up to a constant multiple, 
$\Phi_0 = \prod_{P \in \Gamma} \Phi_P^{m_P}$
where $\Phi_{(a:b)}(X, Y) = (X, Y) \wedge (a, b) = bX - aY$ is the unique linear form vanishing at $(a:b)$, and $m_P = \ord_P(\Phi) \ge 1$ for all $P \in \Gamma$.} 
over $K(T)$
and $\Phi_1$ has no roots in $\P^1(K(T))$,
so that
\[
    V(\Phi)(K(T)) = \Gamma = V(\Phi_0)(K(T)).
\]
In particular, 
\begin{equation} \label{eq:lem2_1}
    V(\Phi_t)(K) = (\Gamma)_t \cup V((\Phi_1)_t)(K)
\end{equation}
for all $t \in U(K)$.
We claim that if $t \in U(K) \setminus Z_l$ 
then 
\begin{equation} \label{eq:lem2_2}
    \Gamma_t \setminus (\Gamma)_t \subseteq f_t^{-\infty}\big(V((\Phi_1)_t)(K)\big).
\end{equation}
To see why, let $Q \in \Gamma_t \setminus (\Gamma)_t$
and consider the following two cases.
\begin{enumerate}[i.]
    \item 
    \textit{Some iterate of $Q$ lies in $(\Gamma)_t$.}
    Let $k$ be such that $f_t^k(Q) \not\in (\Gamma)_t$
    and $f_t^{k+1}(Q) \in (\Gamma)_t$.
    Then $f_t^{k+1}(Q) = P_t$ for some $P \in \Gamma$.
    Specializing the identity $f^{l! + m_0}(P) = f^{m_0}(P)$ yields
    \[
    f_t^{l! + m_0 + 1}(f_t^k(Q)) 
       = f_t^{m_0 + 1}(f_t^k(Q)).
    \]
    \item 
    \textit{No iterate of $Q$ lies in $(\Gamma)_t$.}
    Since $Q$ is $f_t$-preperiodic, 
    there exists $k \ge 0$
    such that $f_t^{m_0 + 1 + k}(Q)$ is $f_t$-periodic.
    The assumption $t \not \in Z_l$ prevents the period of this point from exceeding $l$,
    so that 
    \[
        f_t^{l! + m_0 + 1}(f_t^k(Q)) 
        = f_t^{l!}(f_t^{m_0 + 1 + k}(Q))
        = f_t^{m_0 + 1}(f_t^k(Q)).
    \]
\end{enumerate}
Either way,
some iterate of $Q$ (namely, $f_t^k(Q)$)
lies in $V(\Phi_t)(K) \setminus (\Gamma)_t$, 
hence in $V((\Phi_1)_t)(K)$ by \eqref{eq:lem2_1}.
This proves \eqref{eq:lem2_2}.

Now if $t \in U(K) \cap \Sigma^c \cap Z_l^c$
then 
$\Gamma_t \setminus (\Gamma)_t$ is nonempty,
so by \eqref{eq:lem2_2} $V((\Phi_1)_t)(K)$ is nonempty.
Therefore, 
$\Sigma^c \cap Z_l^c \subseteq U(K)^c \cup \Omega$ 
where 
\[
    \Omega = \{t \in K : V((\Phi_1)_t)(K) \ne \varnothing\}
\]
is the set of parameters where the specialization of $\Phi_1$ acquires a root in $\P^1(K)$.
By Hilbert's irreducibility theorem,
$\Sigma^c \cap Z_l^c$ is thin, as desired.
\end{proof}

\begin{rmk}
Complementary to \eqref{eq:lem2_2} we have 
\[
    V((\Phi_1)_t)(K) \subseteq \Gamma_t \setminus (\Gamma)_t
\]
for all $t \in U(K) \setminus V(\Res(\Phi_0, \Phi_1))(K)$.
This is because 
$V((\Phi_0)_t) \cap V((\Phi_1)_t) = V((\Phi_0)_t, (\Phi_1)_t) \ne \varnothing$ (over $\bar{K}$) iff $\Res((\Phi_0)_t, (\Phi_1)_t) = \Res(\Phi_0, \Phi_1)_t = 0$.
It follows that 
\[
    \Omega \subseteq \Sigma^c \cup V(\Res(\Phi_0, \Phi_1))(K).
\]
\end{rmk}

\begin{q}
Given that 
\[
    Z_0 
    = \{t \in U(K) : \text{$f_t$ has a $K$-cycle}\} 
    = \begin{cases}
    U(K) & \Gamma \ne \varnothing \\
    U(K) \setminus \Sigma & \Gamma = \varnothing
    \end{cases}
\]
we ask: Can $\Gamma$ be empty?
(See also Remark \ref{rmk:ingram_sadek} below.)
\end{q}

Combining Lemmas \ref{lem:inj_locus} and \ref{lem:surj_locus} 
yields the following fundamental estimate for the counting function of the exceptional set.

\begin{cor} \label{cor:exceptional_count}
For each sufficiently large $l$, 
\[
    \Num(E, X) = \Num(Z_l, X) + O\big(\Num(K, X)^{1/2}\big)
\]
where the implicit constant depends on $f$, $l$, and $K$.
\end{cor}

\begin{proof}
By Lemma \ref{lem:inj_locus}, 
$E$ and $\Sigma^c$ differ by a finite set,
while by Lemma \ref{lem:surj_locus}, 
$\Sigma^c$ and $Z_l$ differ by a thin set.
Since
\[
    E \mathbin{\triangle} Z_l 
    = 
    (E \mathbin{\triangle} \Sigma^c) \mathbin{\triangle} (\Sigma^c \mathbin{\triangle} Z_l)
\]
we get that $E$ and $Z_l$ differ by a thin set (as thin sets form an ideal).
Note that 
\[
    \lvert m(A) - m(B) \rvert \le m(A \mathbin{\triangle} B)
\]
for any finite non-negative additive set-function $m$.
So, for each $X$, 
\[
    \Num(E, X) = \Num(Z_l, X) + O\big(\Num(E \mathbin{\triangle} Z_l, X)\big).
\]
The claim now follows from \eqref{eq:thin} and Schanuel's theorem \eqref{eq:schanuel},
which entail
\[
    \Num(\Omega, X) \ll X^n \ll \Num(K, X)^{1/2}
\]
for any thin subset $\Omega$ of a number field of degree $n$.
\end{proof}

\section{Portrait-preserving specializations} \label{sec:little_o}

The main result of this section says that almost all specializations are portrait-preserving.

\begin{thm} \label{thm:little_o}
Let $f_t$ be a one-parameter family of rational maps over a number field $K$ with 
exceptional set $E$. 
Then 
\[
    \Num(E, X) = o\big(\Num(K, X)\big).
\]
\end{thm}

To prove Theorem \ref{thm:little_o}, 
we need a lemma relating bad primes of $f_t$ to roots of the resultant polynomial of the family.
Recall that a polynomial
\[c = \sum_i c_i T^i \in L[T]\]
is called \emph{$v$-integral} if $v(c) = \min_i v(c_i) \ge 0$.
By the nonarchimedean property,
if $c$ is $v$-integral then $v(c(t)) \ge 0$ whenever $v(t) \ge 0$.

\begin{lem} \label{lem:badred}
Let $L$ be a field with discrete valuation $v$ and residue field $\ell$.
Let $f \in L(T)(z)$ have integral lift $F, G \in L[T][X, Y]$ and resultant $R := R_f = \Res(F, G)$. 
Let $t \in L$. 
Write $\tilde{R}$ (resp.~$\tilde{t}$) for the reduction of $R$ (resp.~$t$) mod $v$.
If: 
\begin{enumerate}[(i)]
\item $R(t) \ne 0$,
\item the coefficients of $F$ and $G$ are $v$-integral, and
\item $f_t$ has does not have good reduction at $v$,
\end{enumerate}
then 
$\tilde{t} \in V(\tilde{R})(\ell) \cup \{\infty\}$.
\end{lem}

\begin{proof}
Since $v(t) < 0$ iff $\tilde{t} = \infty$ it suffices to consider $v(t) \ge 0$.
Write
\[
	F = a_d X^d + \ldots + a_0 Y^d \quad \text{and} \quad G = b_d X^d + \ldots + b_0 Y^d
\]
with $a_i, b_i \in L[T]$ for all $i$.
The numbered hypotheses imply, in turn, that 
\begin{enumerate}[(i)]
\item $F_t, G_t$ is a lift of $f_t$, 
\item $v(F_t, G_t) = \min_i \{v(a_i(t)), v(b_i(t))\} \ge 0$, and
\item $v(\Res(F_t, G_t)) = v(\Res f_t) + 2d \, v(F_t, G_t) > 0$
\end{enumerate}
whence $R(t) \equiv 0$ (mod $v$), as desired. 
\end{proof}

We also need a result bounding the maximum cycle length of a rational map in terms of the norm and ramification index of a prime of good reduction.

\begin{propn} \label{propn:n=mrpe}
Let $\phi$ be a rational function of degree at least 2
defined over a number field $K$,
and suppose $\phi$ has good reduction at some prime $\mathfrak{p}$ of $K$.
Then every $K$-rational cycle of $\phi$ has length at most $3 e_{\mathfrak{p}} \Nm \mathfrak{p}^2$.
\end{propn}

\begin{proof}
This follows from \cite[Theorems 2.21 and 2.28]{Silverman}.
\end{proof}

\begin{proof}[Proof of Theorem \ref{thm:little_o}.]
By Corollary \ref{cor:exceptional_count},
\[
    0 \le \limsup_{X \to \infty} \frac{\Num(E, X)}{\Num(K, X)} = \limsup_{X \to \infty} \frac{\Num(Z_l, X)}{\Num(K, X)}
\]
for each $l \gg 0$,
so it suffices to prove the r.h.s.~is arbitrarily small.
To that end, let $\varepsilon > 0$. 
Fix an integral lift $F, G$ of $f$ and put $R = \Res(F, G)$.
Pick a prime $\mathfrak{p}$ of $K$ such that: 
\begin{enumerate}[(i)]
\item the coefficients of $F$ and $G$ are $\mathfrak{p}$-integral,
\item $\tilde{R} \in \F_{\mathfrak{p}}[T]$ is not the zero polynomial, and 
\item $\Nm \mathfrak{p} > \varepsilon^{-1}(\deg R + 1) - 1$.
\end{enumerate}
This is possible because each condition holds for all but finitely many primes $\mathfrak{p}$.
For $l$ sufficiently large (e.g.,~exceeding $3 e_{\mathfrak{p}} \Nm \mathfrak{p}^2$) we have 
\begin{equation} \label{eq:Zl_bad_reduction}
	Z_l 
	\subseteq 
	\{t \in U(K) : f_t \text{ has bad reduction at }\mathfrak{p}\} 
	\subseteq 
	\{t \in K : \tilde{t} \in V(\tilde{R})(\F_{\mathfrak{p}}) \cup \{\infty\}\}. 
\end{equation}
by Proposition \ref{propn:n=mrpe} and Lemma \ref{lem:badred}.
But, by a refinement of Schanuel's theorem\footnote{Communicated to me by V.~Dimitrov.  The idea is that, after embedding diagonally in the adeles, each residue class of $\F_\mathfrak{p}$ becomes a translate of a sublattice of $K$ with index ${\#}\P^1(\F_\mathfrak{p})$.}
\begin{equation} \label{eq:refined_Schanuel}
	\Num(\{t \in K : t \equiv t_0 \text{ (mod $\mathfrak{p})$}\}, X) \sim \frac{\Num(K, X)}{\Nm \mathfrak{p} + 1}
\end{equation} 
for each $\tilde{t_0} \in \P^1(\F_{\mathfrak{p}})$.
In particular, 
when ordered by height, 
the proportion of elements of $K$ lying in a fixed residue class modulo $\mathfrak{p}$ is both independent of the class and decreasing in $\Nm \mathfrak{p}$.
It follows from \eqref{eq:Zl_bad_reduction} and \eqref{eq:refined_Schanuel} 
(and our choice of $\mathfrak{p}$)
that 
\[
	\frac{\Num(Z_l, X)}{\Num(K, X)} \lesssim \frac{\deg R + 1}{\Nm \mathfrak{p} + 1} < \varepsilon. \qedhere
\]
\end{proof}

A natural question is whether Theorem \ref{thm:little_o} can be improved.
In this connection:

\begin{propn} \label{propn:equivalences1}
Let $f_t$ be a one-parameter family of rational maps over a number field $K$ with domain of definition $U$ and exceptional set $E$.
Each of the following statements implies the next; 
moreover (b)iii implies (b)i and (c)ii implies (c)i.

\begin{enumerate}[(a)] 
\item (\emph{uniform boundedness of preperiodic points}) $\sup_{t \in U(K)} \#{\PrePer}(f_t, K) < \infty$.
\item 
\begin{enumerate}[i.]
\item (\emph{uniform boundedness of periodic points}) 
    $\sup_{t\in U(K)} \#{\Per}(f_t, K) < \infty$.
\item (\emph{no large cycles}) 
    $Z_l = \varnothing$ for all $l \gg 0$.
\item (\emph{finitely many large cycles}) 
    $\# Z_l < \infty$ for some $l$.
\end{enumerate}
\item 
\begin{enumerate}[i.]
\item (\emph{rare large cycles}) 
    $\Num(Z_l, X) \ll \Num(K, X)^{1-\delta}$ for some $l, \delta > 0$.
\item (\emph{strong zero-density}) 
    $\Num(E, X) \ll \Num(K, X)^{1-\delta'}$ for some $\delta' > 0$.
\end{enumerate}
\item (\emph{weak zero-density}) $\Num(E, X) = o(\Num(K, X))$.
\end{enumerate}
\end{propn}
\begin{proof}
Trivially, (a) implies (b)i.
Given (b)i,
let $l_0 = \sup_{t \in U(K)} \#{\Per}(f_t, K)$.
If $f_t$ has an $l$-cycle, 
then $l \le \#{\Per}(f_t, K) \le l_0$.
Thus $Z_l = \varnothing$ for all $l > l_0$, which is (b)ii---which obviously implies (b)iii.
To show (b)iii implies (b)i,
note that $U(K) = (U(K) \setminus Z_l) \cup Z_l$ for all $l$.
If $t \in U(K) \setminus Z_l$ then every cycle of $f_t$ has length at most $l$, 
so 
	\[ \#{\Per}(f_t, K) \le \#{\Fix}(f_t^{l!}, K) \le d^{l!} + 1 \] 
since $\deg f_t = \deg f = d > 1$.
It follows that 
	\[ \sup_{t \in U(K)} \#{\Per}(f_t, K) \le \max\{d^{l!} + 1, \sup_{t \in Z_l} \#{\Per}(f_t, K)\}.\]
Thus (b)iii implies (b)i. 
Next, (c)i is immediate from (b) via (b)ii, 
and the equivalence of (c)i and (c)ii follows from Corollary \ref{cor:exceptional_count}: for example,
\[\Num(E, X) \ll \Num(Z_l, X) + \Num(K, X)^{1/2} \ll \Num(K, X)^{\max\{1/2, 1-\delta\}}.\] 
Finally, (c) implies (d) since $\Num(K, X) \to \infty$.
\end{proof}

We present this hierarchy as a way to situate our conjecture relative to other plausible statements and well-known conjectures.
Of these, only (d)
is known to be true---that is our Theorem \ref{thm:little_o}. 
None of the converse implications are known to hold in general, 
but several remarks are in order.

\begin{rmk} \label{rmk:iso} 
    Isotrivial families are uniformly bounded.\footnote{This seems to be folklore, e.g.,~it is stated without proof in \cite[p.~1]{Ingram}.}
    Indeed, if $f = g^\psi$ 
    for some $\psi$ defined over $\Bar{K(T)}$ and some $g$ defined over $\Bar K$, then for all $t$ in $U(K)$, $f_t$ is $\Bar{K}$-conjugate to $g$. (In the notation of the proof of Lemma \ref{lem:inj_locus}, $f_t = g^{\psi_w}$ for any $w \mid t$; and $\psi_w$ is defined over the residue field of $w$, which is a finite extension of $K$).
    By Levy--Manes--Thompson \cite[Theorem 2.9]{LevyManesThompson}, 
    there exists a constant $B_g$ such that $\#{\PrePer}(f_t, K) \le B_g$.
\end{rmk}

\begin{rmk} \label{rmk:ingram_sadek} 
Generalizing Ingram \cite[Theorem 4]{Ingram}, 
Sadek \cite[Theorem 2.3]{Sadek2021}
has given examples of one-parameter families of polynomials of any degree and over any number field $K$ 
with the property that $(\Gamma)_t = \Gamma_t = \{\infty\}$ 
for all $t$ in $K$.
\end{rmk}

\begin{rmk} 
Doyle and Poonen \cite[Corollary 1.9]{DoylePoonen} have shown that, 
for the unicritical family $f_t(z) = z^d + t$ over any number field $K$, 
uniform boundedness of periodic points (b)i implies uniform boundedness of preperiodic points (a).
For the same family, 
Sadek remarked that (d) holds, 
proving it via (c)ii over $K = \Q$ 
\cite[Theorem 4.8]{Sadek2018}.
\end{rmk}

\begin{rmk} 
Poonen's conjecture \cite[Conjecture 2]{Poonen} for the quadratic family 
$z^2 + t$ over $K = \Q$ is the effective ``no large cycles'' statement that $Z_3 = \varnothing$.
The well-known results of Morton \cite[Theorem 4]{Morton} and Flynn--Poonen--Schaefer \cite[Theorem 1]{FPS} (and Stoll \cite[Theorem 7]{Stoll}) say that $Z_3 = Z_4 = Z_5 \, (= Z_6)$.
Hutz and Ingram \cite[Proposition 1]{HutzIngram} showed that $\Num(Z_3, 10^8) = 0$.
\end{rmk}

\begin{rmk} 
Another example of a ``no large cycles'' conjecture was made by Canci and Vishkautsan \cite[Conjecture 1]{CV}, 
who studied
quadratic rational maps with a critical point of period 2.
Working over $\Q$, 
they proved that no such map has a point of period $l$ for $3 \le l \le 6$ 
and conjectured the same for all larger $l$. 
See also Remark \ref{rmk:cv} below.
\end{rmk}

In Section \ref{sec:denominator} we give a sufficient condition for a given family to have the strong zero-density property (c)ii.

\section{Moments of the portrait size} \label{sec:avg}

The main result of this section is a moment-estimate for 
the number of preperiodic points in a one-parameter family, 
along with some criteria for the average value to equal the generic value.

\begin{thm} \label{thm:avg}
Let $f_t$ be a one-parameter family of rational maps over a number field $K$ with domain of definition $U$ and exceptional set $E$,
and let $m$ be a positive integer. Then
\[
    \sum_{\mathclap{\substack{t \in U(K) \\ H(t) \le X}}} \#\Gamma_t^m
    = \#\Gamma^m \cdot \Num(U(K), X)
    + O\big(\Num(E, X) X^{C/\log \log X}\big)
\]
as $X \to \infty$, 
where the constants depend on $f$, $n$, and $m$.
If $f$ is a polynomial, 
the error term may be improved to $O\big(\Num(E, X) (\log X)^m \big)$.
\end{thm}

There are several ingredients in the proof of Theorem \ref{thm:avg}.
After isolating the main term, 
we use the max--length inequality to estimate the remainder term.
The key to the latter is a bound due to Troncoso and Benedetto 
on the number $\#{\PrePer}(f_t, K)$ of $K$-rational preperiodic points of $f_t$ in terms of the number $s = s(f_t)$ of primes of bad reduction of $f_t$. 
By the successive application of three further Lemmas, 
we are then able to bound $s(f_t)$ by a function of $X$.

\begin{propn} \label{propn:benetronc}
Let $\phi$ be a rational function of degree $d \ge 2$ defined over a number field $K$ of degree $n$, and let $s$ be the number of primes of bad reduction of $\phi$ (including all the archimedean ones). 
Then for any $\varepsilon > 0$ there exists $s_0 = s_0(\varepsilon, d, n)$ such that
\[
\#{\PrePer}(\phi, K) \le 
\begin{dcases*} 
    5 \cdot 2^{16 s d^3} + 3 & if $\phi$ is rational and $s \ge 0$ \\
    \Big[\frac{d^2 - 2d + 2}{\log d} + \varepsilon\Big] s \log s    & if $\phi$ is polynomial and $s \ge s_0$
    \end{dcases*}
\]
\end{propn}

\begin{proof}
The general case is due to Troncoso \cite[Corollary 1.3(c)]{Troncoso} 
while the polynomial case is due to Benedetto \cite[Theorem 7.1]{Benedetto}.
\end{proof}

Let $I \subseteq O_K$ be an integral ideal, 
and let $\omega(I) = \omega_K(I)$ denote the number of distinct prime ideals of $K$ dividing $I$.
The following Lemma is a generalization of the classical estimate
\[\omega_\Q(m) \ll \frac{\log m}{\log \log m}\] 
though we state it in a more readily applicable form.
(We were unable to locate a reference.)

\begin{lem} \label{lem:PNT} \label{lem:L}
Let $L(x) = x \log x$ for $x \ge 1$.
\begin{enumerate}[(a)]
\item 
We have 
\[
    L(\omega(I)) \lesssim \log \Nm I
\]
as $\omega(I) \to \infty$.
\item 
The function $L$ is an order-isomorphism $[1, \infty) \to [0, \infty)$.
Moreover,
\[
    L^{-1}(y) \sim \frac{y}{\log y} \qquad (y \to \infty).
\]
In particular, $L^{-1}(ax + b) \sim a L^{-1}(x)$ for all $a > 0$ and all $b$. 
\end{enumerate}
\end{lem}

\begin{proof} \hfill 
\begin{enumerate}[(a)]
\item 
Let $\mathfrak{p}_1, \ldots, \mathfrak{p}_s$ be distinct prime ideals dividing $I$, ordered so that if $\mathfrak{p}_i \cap \Z = p_i\Z$ then $p_1 \le \ldots \le p_s$.
Note that $\Nm \mathfrak{p}_i = p^{f_i} \ge p_i$ for all $i$.
Since every prime number splits into at most $n = [K : \Q]$ prime ideals, 
\[
    p_{(k-1)n+1}, \ldots, p_{kn} \ge \text{$k$\textsuperscript{th} prime}
\]
whenever $kn \le s$. 
Taking $k = \lfloor s/n \rfloor$ we get that 
\[\log \Nm I \ge \log p_1 + \ldots + \log p_s \ge n \cdot \vartheta(\text{$k$\textsuperscript{th} prime)}.\]
By the prime number theorem, 
\[\vartheta(\text{$k$\textsuperscript{th} prime}) \sim \text{$k$\textsuperscript{th} prime} \sim k \log k\]
as $k \to \infty$.
Therefore 
\[\log \Nm I \gtrsim n \cdot \lfloor s/n \rfloor \log \lfloor s/n \rfloor \sim s \log s\]
as $s = \omega(I) \to \infty$.
\item 
$L(1) = 0$ and $L(x) \to \infty$ with $x$ (so $L$ is onto),
and $L'(x) = 1 + \log x > 0$ for $x \ge 1$ (so $L$ is strictly increasing).
Finally, if $y = L(x)$ then 
\[
    \frac{L^{-1}(y) \log y}{y}
    = 
    \frac{x (\log x + \log \log x)}{x \log x} 
    = 
    1 + o(1). \qedhere
\]
\end{enumerate}
\end{proof}

The next Lemma bounds the norm of the resultant ideal in terms of the coefficients of the morphism.

\begin{lem} \label{lem:Discheight}
Let $\phi : \P^1 \to \P^1$ be a morphism of degree $d$ defined over $K$ with lift 
\[
    F(X, Y) = F_d X^d + \dots + F_0 Y^d, \, G(X, Y) = G_d X^d + \dots + G_0 Y^d \in K[X, Y].
\]
Let $\Res \phi$ be the resultant ideal of $\phi$ 
and let 
\[H(\phi) = H(F_d : \ldots : F_0 : G_d : \ldots : G_0)\]
be the height of $\phi$.
Then 
\[
    \Nm \Res \phi \le C H(\phi)^{2dn}
\]
where $C = (ab)^{dn/2}$, with $a$ (resp.~ $b$) being equal to the number of nonzero coefficients of $F$ (resp.~$G$).
If $\phi$ is totally ramified at $0$ or $\infty$, we may take $C = 1$.
\end{lem}
\begin{proof}
Let $M_K$ be the set of places of $K$ extending the standard places of $\Q$.
Write $\lvert F, G \rvert_v = \max_{i,j} \{ \lvert F_i \rvert_v, \lvert G_j \rvert_v\}$ for each $v \in M_K$.
By unique factorization of ideals,
\[
    \Res \phi = \Res(F, G) O_K 
        \cdot \prod_{\mathfrak{p}} \mathfrak{p}^{-2d \ord_{\mathfrak{p}}(F,\,G)}
\]
whence
\begin{equation} \label{eq:NmResf1}
    \Nm \Res \phi = \prod_{v \mid \infty} \lvert {\Res}(F, G) \rvert_v^{n_v} \cdot \prod_{v \nmid \infty} \lvert F, G \rvert_v^{2d n_v}.
\end{equation}
This is because 
\[
    \Nm \alpha O_K = \lvert N_{K/\Q}(\alpha) \rvert = \! \! \prod_{\sigma : K \into \C} \! \! \lvert \sigma(\alpha) \rvert = \prod_{v \mid \infty} \lvert \alpha \rvert_v^{n_v}
\]
for all $\alpha \in K$ (regardless of whether $K/\Q$ is Galois)~while 
\[
    \Nm \mathfrak{p}^{-{\ord}_\mathfrak{p}(\beta)} = \lvert \beta \rvert_\mathfrak{p}^{n_\mathfrak{p}} 
\]
for all $\beta \in K$.
On the other hand, since 
\[ H(\phi) = 
\prod_{v \in M_K} \lvert F, G \rvert_v^{n_v/n} \]
we just have to bound the archimedean contribution.

Let $\sigma : K \into \C$ be an embedding corresponding to the infinite place $v$ of $K$ 
and consider the Sylvester matrix of 
\[F^\sigma(X, Y) = \sigma(F_d) X^d + \ldots + \sigma(F_0) Y^d \ \text{and} \ G^\sigma(X, Y) = \sigma(G_d) X^d + \ldots + \sigma(G_0) Y^d.\]
Its first $d$ columns 
each have $a$ nonzero entries, 
and $\lvert \sigma(F_i) \rvert = \lvert F_i \rvert_v \le \lvert F \rvert_v$ for all $i$, 
so
\[
    \sqrt{\lvert \sigma(F_d) \rvert^2 + \ldots + \lvert \sigma(F_0) \rvert^2} \le \sqrt{a} \cdot \lvert F \rvert_v.
\]
Likewise, the 2-norm of each of the last $d$ columns is bounded by $\sqrt{b} \cdot \lvert G \rvert_v$.
Thus, the Leibniz formula and Hadamard's inequality\footnote{If $A$ is a square matrix over $\C$ with columns $A_1, \ldots, A_k$ then $\det A \le \lVert A_1 \rVert \ldots \lVert A_k \rVert$ where $\lVert x \rVert = \sqrt{\langle x, x \rangle}$ denotes the 2-norm.} yield 
\[
    \lvert \Res(F, G) \rvert_v 
    = \lvert \Res(F^\sigma, G^\sigma) \rvert 
    \le (\sqrt{a} \cdot \lvert F \rvert_v)^d (\sqrt{b} \cdot \lvert G \rvert_v)^d 
    = (ab)^{d/2} \, \lvert F \rvert_v^d \, \lvert G \rvert_v^d.
\]
If $\phi$ is totally ramified at 0 (resp.~$\infty$) 
then the Sylvester matrix is upper (resp.~lower) triangular.
In this special case, 
\[
    \lvert \Res(F, G) \rvert_v = \lvert F_d\rvert_v^d \, \lvert G_0 \rvert_v^d \le \lvert F \rvert_v^d \, \lvert G \rvert_v^d.
\]
Since $\displaystyle \lvert F \rvert_v^d \, \lvert G \rvert_v^d \le \lvert F, G \rvert_v^{2d}$ we obtain, in either case,
\begin{equation} \label{eq:NmResf2}
    \prod_{v \mid \infty} \lvert \Res(F, G) \rvert_v^{n_v} \le C \prod_{v \mid \infty} \lvert F, G \rvert_v^{2d n_v}.
\end{equation}
Combining \eqref{eq:NmResf1} and \eqref{eq:NmResf2} proves the claim.
\end{proof}

The bound in Lemma \ref{lem:Discheight} is attained infinitely often for polynomials over any number field,
as witnessed by the maps $\phi(z) = z^d + 1/b$ where $b = 1, 2, 3, \ldots$.
On the other hand, 
there exist maps $\phi$ of arbitrarily large height with $\Nm \Res \phi = 1$ (e.g.,~monic polynomials with algebraic integer coefficients).

\begin{q}
Fix a number field $K$ of degree $n$ and a positive integer $d$.
What is the set of subsequential limits of
\[\frac{\Nm \Res \phi}{H(\phi)^{2dn}}\]
as $\phi$ ranges over rational maps of degree $d$ defined over $K$?
In a similar vein, what is
\[P(c) := \lim_{X \to \infty} \frac{\#\{\phi \in \Rat_d(K) : H(\phi) \le X \, \text{and} \, \Nm \Res \phi \le c H(\phi)^{2dn}\}}{\Num(\Rat_d(K), X)}\]
as a function of $c \in \R$?
Note that $P(c)$ is nondecreasing, 0 for $c \le 0$, and 1 for $c \ge (d+1)^{dn}$ by Lemma \ref{lem:Discheight}. Is $P$ right-continuous?
\end{q}

The final link is provided by the next Lemma, which is completely standard.

\begin{lem} \label{lem:Hft_to_Ht}
Let $f_t$ be a one-parameter family of degree-$d$ rational maps over a number field $K$ with domain of definition $U$ and co-degree $d'$ (cf.~Remark \ref{rmk:codegree}). 
Then there exist constants $C \ge c > 0$ depending on $f$ such that 
\[cH(t)^{d'} \le H(f_t) \le CH(t)^{d'}\]
for all $t \in U(K)$.
\end{lem}
\begin{proof}
The family induces a degree-$d'$ rational map
\begin{align*}
    \A^1 &\dashrightarrow \Rat_d \\ 
    t &\mapsto f_t
\shortintertext{which extends to a morphism}
    \P^1 &\to \bar\Rat_d \cong \P^{2d+1}.
\end{align*}
The claim follows from basic properties of heights \cite[Chapter 4, Theorem 1.8]{Lang}. 
\end{proof}

\begin{proof}[Proof of Theorem \ref{thm:avg}]
Write $\mathcal{A}_m(X)$ for the sum in question.
Since $\Gamma_t \cong \Gamma$ on $U(K) \setminus E$
and $\Num(U(K) \setminus E, X) = \Num(U(K), X) - \Num(E, X)$,
we have
\[
   \mathcal{A}_m(X) 
    = \#\Gamma^m \cdot \Num(U(K), X) + O\big(\Num(E, X)\big) 
    \ + \!\! \sum_{\substack{t \in E \\ H(t) \le X}} \!\! \#\Gamma_t^m.
\]
Trivially,
\begin{equation} \label{eq:Rbound}
    0 
    \le 
    \sum_{\substack{t \in E \\ H(t) \le X}} \!\! \#\Gamma_t^m 
    \le 
    \Num(E, X) \Big(\max_{H(t) \le X} \# \Gamma_t\Big)^m
\end{equation}
so it remains to estimate the maximum.

Let $s = s(f_t)$ be the number of bad primes of $f_t$ (including all archimedean ones) (i.e.,~places of bad reduction). 
By Proposition \ref{propn:benetronc}, 
\[
    \#\Gamma_t \ll \begin{dcases*} 
    C_1^s       & if $f$ is rational \\
    s \log s    & if $f$ is polynomial
    \end{dcases*}
\]
where $C_1 = 2^{16d^3}$ and the implied constants depend on $d$.
By definition, $s = s_\infty + \omega(\Res f_t)$ where $s_\infty$ is the number of archimedean places of $K$.
Chaining together Lemmas \ref{lem:PNT}, \ref{lem:Discheight}, and \ref{lem:Hft_to_Ht} 
shows that there exist constants $s_0$ and $C_2, C_3 > 0$ 
depending on $f$ and $n$ 
such that 
\begin{equation} \label{eq:slogs}
s \log s \le C_2 h(t) + C_3
\end{equation}
whenever $s > s_0$. 
Indeed, 
\[
    s \log s = L(s) \sim L(s - s_\infty) \lesssim \log \Nm \Res f_t
\]
as $s \to \infty$; 
moreover, 
\[
    \log \Nm \Res f_t \le 2dn \, h(f_t) + O(1);
\]
finally,
\[ 
    2dn \, h(f_t) \le 2dd'n \, h(t) + O(1).
\]
Together these give \eqref{eq:slogs}.
This and \eqref{eq:Rbound} already settles the polynomial case.
In the general case, Lemma \ref{lem:L}(b) implies 
\[
    \max_{H(t) \le X} s(f_t) \le \max\{s_0, L^{-1}(C_2 \log X + C_3)\} \le C_4 \frac{\log X}{\log \log X}
\]
for any $C_4 > C_2$ and $X$ sufficiently large.
Theorem \ref{thm:avg} follows on exponentiating:
\[
    \max_{H(t) \le X} C_1^{s(f_t)} \le C_1^{C_4 \log X / \log \log X} = X^{C_4 \log C_1 / \log \log X}. \qedhere
\]
\end{proof}

The first moment $\mathcal{A}(X) := \mathcal{A}_1(X)$ 
and the remainder term 
\[
    \mathcal{R}(X) := \mathcal{A}(X) - \#\Gamma \cdot \Num(U(K), X) = \! \! \sum_{\substack{t \in U(K) \\ H(t) \le X}} \! \! (\#\Gamma_t - \#\Gamma)
\]
are of inherent interest;
the former counts the total number of preperiodic points in the family, 
while the latter may be interpreted as the number of ``unexpected'' preperiodic points.
As with Proposition \ref{propn:equivalences1} regarding $\Num(E, X)$, 
we present a hierarchy of statements relating the growth rate of $\mathcal{R}(X)$ to other 
properties of the family.

\begin{propn} \label{propn:equivalences2}
Let $f_t$ be a one-parameter family of rational maps over a number
field $K$ with exceptional set $E$.
Each of the following statements implies the next;
moreover (x)ii implies (x)i for each x = a, b, c.
\begin{enumerate}[(a)]
    \item 
    \begin{enumerate}[i.]
        \item $\mathcal{R}(X) = O(1)$.
        \item $\Num(E, X) = O(1)$.
    \end{enumerate}
    \item 
    \begin{enumerate}[i.]
        \item $\mathcal{R}(X) \ll \Num(K, X)^{1-\varepsilon}$ for some $\varepsilon > 0$.
        \item $\Num(E, X) \ll \Num(K, X)^{1-\delta}$ for some $\delta > 0$.
    \end{enumerate}
    \item 
    \begin{enumerate}[i.]
        \item $\mathcal{R}(X) = o(\Num(K, X))$. 
        \item $\mathcal{A}(X) \sim \#\Gamma \cdot \Num(K, X)$. 
    \end{enumerate}
    \item $\mathcal{A}(X) = o(X^{2n + C/\log\log X})$ for some $C > 0$.
\end{enumerate}
\end{propn}
\begin{proof}
The sum defining $\mathcal{R}$ is supported on $E$, so (a)ii implies (a)i.
Since $\#\Gamma_t > \#\Gamma$ if and only if $t \in E \cap U^\circ(K)$, 
we have 
\[\Num(E, X) \le \mathcal{R}(X) + O(1)\]
by Lemma \ref{lem:inj_locus}.
Thus (a)i implies (a)ii 
and (b)i implies (b)ii with $\varepsilon = \delta$.
Conversely, 
suppose
    \[
    \Num(E, X) \ll \Num(K, X)^{1-\delta}.
    \]
Note that by Schanuel's theorem,
    \[
    \Num(K, X) \asymp X^{2n}
    \]
while by Theorem \ref{thm:avg},
\[ \mathcal{R}(X) \ll \Num(E, X) X^{C/\log\log X} \]	
for some $C > 0$.
So, if $\varepsilon$ is arbitrary 
then the three equations just displayed
imply 
\[
\frac{\mathcal{R}(X)}{\Num(K, X)^{1-\varepsilon}} 
\ll X^{2n(\varepsilon - \delta) + C/\log\log X}.
\]
The r.h.s.~ is bounded iff the exponent is eventually non-negative, 
iff $\varepsilon < \delta$ (since $C > 0$).
Thus (b)ii implies (b)i.
Obviously (b)i implies (c)i,
and by definition
\[
    \lim_{X \to \infty} \frac{\mathcal{A}(X)}{\Num(K, X)} = \#\Gamma + \lim_{X \to \infty} \frac{\mathcal{R}(X)}{\Num(K, X)}
\]
so (c)i and (c)ii are equivalent.
Of course, (c)ii implies (d). 
Finally, (d) always holds by Theorems \ref{thm:little_o} and \ref{thm:avg} and the fact that 
$E \subseteq U(K) \subseteq K$.
\end{proof}

\begin{rmk}
Assuming uniform boundedness of preperiodic points, 
it's possible to write down an explicit formula for $\mathcal{R}(X)$ (and hence for $\mathcal{A}(X)$), 
at least in principle.
Indeed, if $\sup_{t \in U(K)} \#{\PrePer}(f_t, K) < \infty$ 
then the set $\mathcal{C}$ 
of isomorphism classes of portraits arising among the $\PrePer(f_t, K)$'s is finite. 
Restricting $t$ to $U^\circ(K)$, 
each $\mathcal{P} \in \mathcal{C}$ admits an embedded copy of $\Gamma$.
Let $E_\mathcal{P} = \{t \in U^\circ(K) : \mathcal{P} \into \Gamma_t\}$
and note that $E_{\mathcal{P}'} \subseteq E_\mathcal{P}$ whenever $\mathcal{P} \into \mathcal{P}'$.
By the inclusion-exclusion principle, 
there exist integers $n_\mathcal{P}$ such that 
\begin{equation}
\mathcal{R}(X)
= \sum_{\mathcal{P} \in \mathcal{C}} n_\mathcal{P} \cdot \Num(E_\mathcal{P}, X).
\end{equation}
If $t \in E_\mathcal{P} \cap \Sigma$ 
then $\mathcal{P} = \Gamma$, because 
$\mathcal{P} \into \Gamma_t \cong \Gamma \into \mathcal{P}$.
In particular, 
$n_\Gamma = 0$; moreover,
$\Num(E_\mathcal{P}, X) \ll \Num(K, X)^{1/2}$ for all $\mathcal{P} \ne \Gamma$
by Corollary \ref{cor:exceptional_count} and the uniform boundedness assumption.
While this already shows 
\[
\mathcal{R}(X) \ll \Num(K, X)^{1/2}
\]
we can do better.
Each edge $P \to Q$ of a given portrait $\mathcal{P}$ can be encoded as an algebraic condition (\textit{viz.}~$f(P) \wedge Q = 0$),
and imposing all of these defines a curve $\pi : C_\mathcal{P} \to \P^1$ such that $\pi(C_\mathcal{P}(K)) = E_\mathcal{P}$ (perhaps up to a finite set).
Moreover, $\deg \pi \ge 2$ for $\mathcal{P} \ne \Gamma$.
Thus the evaluation of $\mathcal{R}(X)$ reduces to counting rational points on curves.
In Section \ref{sec:quadratic} we carry out this computation for the quadratic family $f_t(z) = z^2 + t$ over $K = \Q$ 
and obtain an explicit estimate for $\mathcal{R}(X)$ assuming Poonen's conjecture.
\end{rmk}

\section{Denominator lemmata} \label{sec:denominator}

In this section 
we give a sufficient condition for a given family to have the strong
zero-density property of Proposition \ref{propn:equivalences1}(c)ii.
Our ideas are inspired by the observation of Walde and Russo that if $z^2 + t$ has a rational preperiodic point distinct from $\infty$ then the denominator of $t$ is a perfect square \cite[Corollary 4]{WaldeRusso}. 
Their result may be generalized in various ways (e.g.,~\cite[Lemma 5(C)]{Ingram}), 
but for our purposes, the essential restriction is captured by the following definition.

\begin{defn}
    A one-parameter family $f_t \in K(z)$ is said to have a \emph{denominator lemma} 
    if there exists a finite set $S$ of places $v$ of $K$ 
    with the property that 
    if $(\Gamma)_t \ne \Gamma_t$ for some $t \in U(K)$ then $v(t) \ne -1$ for all $v \not \in S$.
\end{defn}

\begin{thm} \label{thm:den_gives_conj}
    Let $f_t \in K(z)$ be a one-parameter family of rational maps
    with exceptional set $E$.
    Suppose this family has a denominator lemma.
    Then for each $\varepsilon > 0$ we have
    \[
        \Num(E, X) \ll \Num(K, X)^{3/4 + \varepsilon}
    \]
    If $K$ is $\Q$ or an imaginary quadratic field, 
    then we may take $\varepsilon = 0$.
\end{thm}

The significance of Theorem \ref{thm:den_gives_conj} is apparent from Propositions \ref{propn:equivalences1} and \ref{propn:equivalences2}: 
If a family has a denominator lemma, 
then it has the strong zero-density property, 
so the average number of preperiodic points 
is equal to the generic number of preperiodic points.

\begin{lem} \label{lem:Heights} 
Let $K$ be a number field.
\hfill 
\begin{enumerate}[(a)] 
\item $H(xy) \le H(x) H(y)$ for all $x, y \in K$
\item $H(t^{-1}) = H(t)$ for all $t \in K^\times$
\item Suppose $tO_K = \mathfrak{a} \mathfrak{b}^{-1}$ where $\mathfrak{a}$ and $\mathfrak{b}$ are coprime integral ideals (i.e.,~the \emph{numerator} and \emph{denominator} ideals of $t$).
Then 
$H(t) \ge \max\{\Nm \mathfrak{a}, \Nm \mathfrak{b}\}^{1/n}$, 
with equality if $K = \Q$.
\end{enumerate}
\end{lem}
\begin{proof}
These may be gleaned from the first few pages of Lang \cite[Chapter 3]{Lang}.
\end{proof}

\begin{lem} \label{lem:LangWeberGolomb}
Let $S$ be a finite set of prime ideals in a number field $K$.
Put
\[
    E^* = \{t \in K : \ord_{\mathfrak{p}}(t) \ne {-1} \text{ for all } \mathfrak{p} \not \in S\}.
\]
Then
\[
    \Num(E^*, X) \ll X^{3n/2} (\log X)^r
\]
where $n = [K : \Q]$, $r = \rk O_K^\times$, 
and the implied constant depends on $\#S$ and $K$.
\end{lem}

\begin{proof}

\noindent \textit{Step 1: Units.}
Let $\mathcal{I}_K$ denote the set of integral ideals of $K$. 
Then the map $\psi : K^\times \to \mathcal{I}_K \times \mathcal{I}_K$
assigning to $t$ its numerator and denominator ideals 
is generally not finite-to-one; 
but if $tO_K = t'O_K$ with $H(t), H(t)' \le X$ and $tt' \ne 0$,
then $t' = tu$ for some $u \in O_K^\times$ 
with 
\[H(u) = H(t^{-1} t') \le H(t^{-1}) H(t') = H(t) H(t') \le X^2\]
by Lemma \ref{lem:Heights}(a) and (b).
Writing $E^*(X)$ for the set of elements of $E^*$ up to height $X$, it follows that
\[
\Num(E^*, X) \le 1 + \# \psi(E^*(X)) \cdot \Num(O_K^\times, X^2).
\]
(accounting for $t = 0$).
By Lang's estimate \cite[Chapter 3, Theorem 5.2(ii)]{Lang}, we have
\[ 
    \Num(O_K^\times, X^2) \ll (\log X)^r, 
\]
so it remains to show that the cardinality of $\psi(E^*(X))$ is $O(X^{3n/2})$.

\medskip 

\noindent \textit{Step 2: Ideals.}
Let $t \in E^*(X)$ and write $\psi(t) = (\mathfrak{a}, \mathfrak{b})$.
Then $\Nm \mathfrak{a}$ and $\Nm \mathfrak{b}$ do not exceed $X^n$ by Lemma \ref{lem:Heights}(c).
Moreover, 
\[
    \ord_{\mathfrak{p}}(\mathfrak{b}) = \max \{0, -{\ord_{\mathfrak{p}}(t)}\} \ne 1 
\]
for all $\mathfrak{p} \not \in S$.
With $\mathcal{B}_S \subset \mathcal{I}_K$ denoting the set of all such ideals, 
and counting by norm,
we have (by unique factorization)
\[
    \Num(\mathcal{B}_S, X) = \sum_{\mathfrak{d} \mid \mathfrak{c}} \Num\Big(\mathcal{B}, \frac{X}{\Nm \mathfrak{d}}\Big) \le 2^{\# S} \Num(\mathcal{B}, X)
\]
where $\mathfrak{c}$ is the product of the elements of $S$
and $\mathcal{B} := \mathcal{B}_\varnothing$.
Imitating Golomb \cite{Golomb}, 
one can show that every element $\mathfrak{b}$ of $\mathcal{B}$ factors uniquely as 
$\mathfrak{b} = \mathfrak{n}^2 \mathfrak{m}^3$ with $\mathfrak{m}$ squarefree.
This yields the formula
\begin{equation} \label{eq:BXformula}
    \Num(\mathcal{B}, X) = \sum_{\mathfrak{m}} \lvert \mu_K(\mathfrak{m}) \rvert \cdot \Num\Big(\mathcal{I}_K, \sqrt{\frac{X}{\Nm \mathfrak{m}^3} }\Big)
\end{equation}
where $\mu_K$ is the M\"obius function of $K$.
From the classical estimate \cite[Satz 202]{Landau} 
\[ 
    \Num(\mathcal{I}_K, X) \ll X
\]
we derive 
\[
    \Num(\mathcal{B}, X) \ll \sum_{\mathfrak{m}} \frac{1}{\Nm \mathfrak{m}^{3/2}} \cdot X^{1/2}
\]
which is $O(X^{1/2})$ as the series is convergent ($3/2 > 1$).
Putting it all together, 
we conclude 
\[
    \# \psi(E^*(X)) \le \Num(\mathcal{I}_K, X^n) \cdot \Num(\mathcal{B}, X^n) \ll X^{3n/2}. \qedhere
\]
\end{proof}

\begin{rmk}
Using the full strength of Weber's theorem \eqref{eq:weber}
\[
    \Num(\mathcal{I}_K, X) = \rho_K X + O(X^{1-1/n})
\]
it follows from \eqref{eq:BXformula} and Abel summation that
\[
    \Num(\mathcal{B}, X) = \rho_K \frac{\zeta_K(3/2)}{\zeta_K(3)} X^{1/2} + O\begin{cases}
    X^{1/3}         & n < 3 \\ 
    X^{1/3} \log X  & n = 3 \\
    X^{(1-1/n)/2}   & n > 3.
    \end{cases}
\]
This generalizes the asymptotics obtained by Golomb for $K = \Q$,
and can also be used to show that Lemma \ref{lem:LangWeberGolomb} is sharp when $r = 0$.
\end{rmk}

\begin{proof}[Proof of Theorem \ref{thm:den_gives_conj}]
Since the surjectivity and bijectivity loci differ by a finite set, 
\[
    \Num(E, X) = \Num(U(K) \setminus \Sigma, X) + O(1).
\]
By hypothesis and Lemma \ref{lem:LangWeberGolomb},
\[
    \Num(U(K) \setminus \Sigma, X) \le \Num(E^*, X) \ll (X^{2n})^{3/4} (\log X)^r
\]
where $r = \rk O_K^\times$.
Note that $r = 0$ iff $K$ is $\Q$ or an imaginary quadratic field.
In any case,
the claim follows from Schanuel's estimate $X^{2n} \ll \Num(K, X)$.
\end{proof}

\section{Tropical dynamics} \label{sec:examples}

In this section we exhibit several classes of one-parameter families admitting denominator lemmata (defined in the previous section), 
thereby giving evidence for the strong zero-density conjecture.
The method ultimately has its origins in \cite[Theorem 6]{WaldeRusso}. 
The general idea is to factor $f$ through the valuation to obtain a piecewise linear model with simpler dynamics.


\begin{eg} \label{eg:1}
Let $L$ be a field with discrete valuation $v$.
Let $d \ge 2$ be an integer and let 
\[
	f(z) = a_d z^d + \ldots + a_1 z + a_0 + c \in L[z]
\]
with $a_d \ne 0$.
Assume $v(a_i) \ge 0$ for all $i$ and $v(a_d) = 0$.
If $v(c)$ is negative and not divisible by $d$, 
then 
\[\PrePer(f, L) = \{\infty\}.\]
\end{eg}

\begin{proof}
It's clear that $f^{-1}(\infty) = \{\infty\}$, 
so let $z \in L$.
By the nonarchimedean property of the valuation, 
\[
	v(f(z)) \ge \min\{v(a_d z^d), \ldots, v(a_1 z), v(a_0 + c)\}
\]
with equality if the minimum is unique,
i.e., if there is a single term with least valuation.
Of course, $v(a_d z^d) = d v(z)$ because $v(a_d) = 0$, 
and $v(a_0 + c) = v(c)$ because $v(a_0) \ge 0 > v(c)$.

For $0 < i < d$, 
we have $v(a_i z^i) = v(a_i) + i v(z) \ge i v(z)$.
If $v(z) < 0$ then $i v(z) > d v(z) = v(a_d z^d)$
while if $v(z) \ge 0$ then $iv(z) \ge 0 > v(c) = v(a_0 + c)$.
Thus, the minimum of 
\[
	v(a_d z^d), \ldots, v(a_1 z), v(a_0 + c)
\]
is either the first term or the last, 
depending on whether $v(z) < 0$ or $v(z) \ge 0$. 
Since $v(c)$ is not divisible by $d$, 
and since the terms we eliminated were all strictly greater 
than the candidate minima, 
it follows that 
\[
	v(f(z)) = \min\{ dv(z), v(c) \} \tag{*}
\]
for all $z$ in $L$.

From (*), 
we get that $v(z) \ge 0$ implies $v(f(z)) = v(c) < 0$,
while $v(z) < 0$ implies $v(f(z)) = dv(z) < v(z)$, as $d > 1$.
Therefore, $v(f^n(z)) \to -\infty$ as $n \to \infty$ 
so $z$ is not preperiodic.
\end{proof}


\begin{eg} \label{eg:2}
Let $L$ be a field with discrete valuation $v$.
Let $d \ge 2$ and $1 \le e \le d$ be integers and let 
\[
	f(z) = \frac{c}{a_e z^e + \ldots + a_d z^d} \in L(z)
\]
with $a_d a_e \ne 0$.
Assume $v(a_i) \ge 0$ for all $i$ 
and $v(a_d) = v(a_e) = 0$.
If $v(c)$ is negative and not divisible by $d$ nor by $d+1$, 
then 
\[\PrePer(f, L) = \{0, \infty\} \cup f^{-1}(\infty).\]
\end{eg}

\begin{proof}
It's clear that $f$ interchanges $0$ and $\infty$, 
and that $f^{-1}(0) = \{\infty\}$,
so let $z \in L$ and suppose $f(z) \ne \infty$.
By the nonarchimedean property,
\[
	v(f(z)) = v(c) - v(a_e z^e + \ldots + a_d z^d)
	        \le v(c) - \min\{v(a_e z^e), \ldots, v(a_d z^d)\}
\]
with equality if the minimum is unique.
As in Example \ref{eg:1}, 
for $e < i < d$ we have $v(a_i z^i) = v(a_i) + i v(z) \ge i v(z)$;
if $v(z) < 0$ then $i v(z) > d v(z) = v(a_d z^d)$
and 
if $v(z) > 0$ then $i v(z) > e v(z) = v(a_e z^e)$.
Thus, if $v(z) \ne 0$ 
then 
\[
	v(f(z)) = v(c) - \min\{d v(z), e v(z)\}
\]
while if $v(z) = 0$ 
then 
\[
	v(f(z)) \le v(c).
\]

Now, suppose $v(z) > 0$.
Then $v(f(z)) = v(c) - e v(z) < v(c) < 0$, 
so \[v(f^2(z)) = (1 - d) v(c) + de v(z)\]
and this exceeds $v(z)$ because $(1 - d) v(c) > 0$ and $de > 1$. 
Thus, $v(f^n(z)) \to +\infty$ as $n \to \infty$, 
so $z$ is not preperiodic.
Similarly, if $v(z) = 0$, 
then $v(f^2(z)) \ge (1 - d) v(c) > 0$,
so $z$ is not preperiodic.

Finally, suppose $v(z) < 0$.
Then $v(f(z)) = v(c) - d v(z)$, which is nonzero because $d \nmid v(c)$. 
It's positive iff $v(z) < v(c)/d$, 
in which case $z$ is not preperiodic by the previous paragraph,
so suppose $v(z) > v(c)/d$.
Then
\[
	v(f^2(z)) - v(z) = (d^2 - 1) \big[ v(z) - v(c)/(d+1) \big]
\]
Since $d \ge 2$ and $d + 1 \nmid v(c)$, the r.h.s.~is never zero.
Thus, 
\begin{align*}
v(z) < v(c)/(d+1) &\implies v(f^2(z)) < v(z), \text{ and} \\
v(z) > v(c)/(d+1) &\implies v(f^2(z)) > v(z).
\end{align*}
Either way, 
$v(f^{2n}(z)) \not \in (v(c)/d, 0)$ for some $n \ge 1$,
so $z$ is not preperiodic.
\end{proof}

\begin{rmk}
The same argument gives the same conclusion for 
\[
    f(z) = \frac{c z^d}{a_0 + \ldots + a_e z^e}
\]
assuming 
$v(a_i) \ge 0$, 
$v(a_e) = v(a_0) = 0$, 
$d > e$, 
and $v(c)$ is negative and not divisible by $d$ nor $d-1$.
\end{rmk}

Our final Example features a 3-cycle.
The novel idea is to keep track of multiple $v$-adic distances simultaneously.

\begin{eg}\label{eg:3}
Let $L$ be a field with discrete valuation $v$, and let $d \ge e+2 \ge 5$ be integers with $v(d) = 0$.
Consider the rational map
\[
	f(z) = \frac{cz^e + 1}{(1 - z)^d} \in L(z).
\]
If $v(c) = -1$ then 
\[\PrePer(f, L) = \{0, 1, \infty\}.\]
\end{eg}

\begin{proof}
It's clear that $0 \mapsto 1 \to \infty \to 0$ under $f$, 
with $f^{-1}(\infty) = \{1\}$, 
so suppose $z \in L \setminus \{0, 1, \infty\}$ 
and let $v(z) = a$ and $v(z - 1) = b$.
Note that $a$ and $b$ are not independent;
if $a > 0$ then $b = 0$, 
and if $b > 0$ then $a = 0$, 
while if $a < 0$ or $b < 0$ then $a = b$. 

Let $a' = v(f(z))$ and $b' = v(f(z) - 1)$.
Since $ea \ne 1$, 
we have 
\[
a' = \min\{ea - 1, 0\} - db
 = 
	\begin{cases}
	0              &  a > b = 0 \\  
	-1 - db        &  a = 0 < b \\  
	-1 - (d - e)a  &  a = b < 0 \\  
	-1             &  a = b = 0.    
	\end{cases}
\]
We claim that 
\[
b' = 
	\begin{cases}
	a              &  a > b = 0 \\  
	-1 - db        &  a = 0 < b \\  
	0              &  a = b < 0 \\  
	-1             &  a = b = 0.    
	\end{cases}
\]
Indeed, the hypothesis that $d - e \ge 2$ implies $a' \ne 0$, 
which completely determines $b'$ in every case but the first.
Hence, suppose $a > b = 0$.
By the binomial theorem, 
\begin{equation} \label{eq:b_prime}
	b' 
	= v(cz^e + 1 - (1 - z)^d) 
	= v\big(cz^e + dz - \sum_{i > 1} \binom{d}{i}(-z)^i \big).
\end{equation}
But 
\[
	v\big(\binom{d}{i}(-z)^i \big) \ge i v(z) > v(dz)
\]
for all $i > 1$; 
likewise, 
\[
    v(cz^e) = ea - 1 > a
\] 
since $e \ge 3$ and $a \ge 1$.
Therefore, $dz$ is the unique term in \eqref{eq:b_prime} with minimal valuation.
Thus $b' = a$, as desired.

Considering the induced action of $f$ on pairs $(a, b) \to (a', b')$ 
we see that 
\begin{itemize}
    \item the origin $a = b = 0$ is mapped to the point $a = b = -1$
    \item the diagonal ray $a = b < 0$ is mapped to the horizontal ray $a' > 0 = b'$
    \item the horizontal ray $a > 0 = b$ is mapped to the vertical ray $a' = 0 < b'$.
\end{itemize}
On the vertical ray $a = 0 < b$
we have
\begin{itemize}
    \item $a' = b' = -1 - db < 0$
    \item $a'' = -1 - (d - e)(-1 - db) > 0$ and $b'' = 0$
    \item $a''' = 0$ and $b''' = -1 - (d - e)(-1 - db)$
\end{itemize}
and $b''' > b$ because \textit{a fortiori}
\[
	b > -\frac{d-e-1}{d(d-e) - 1}.
\]
Thus, if $v(z - 1) > 0$ then $v(f^{3n}(z) - 1) \to +\infty$ as $n \to \infty$.
It follows that no $z$ in $L$ is preperiodic.
\end{proof}

\begin{rmk}
Example \ref{eg:3} extends to the case $e = 2$ and $d \ge 5$, 
with the modified conclusion that
\[
    \PrePer(f, L) = f^{-1}(\{0, 1, \infty\}).
\] 
Indeed, if $v(z) = 1$ then $v(dz) = v(cz^2)$ so the expansion in \eqref{eq:b_prime} does not have a unique minimum.
Thus $v(z) = 1$ merely implies $v(f(z) - 1) \ge 1$.
Either $v(f(z) - 1) = +\infty$, in which case $z \in f^{-1}(1)$,
or $v(f(z) - 1) < +\infty$, 
in which case the rest of the argument goes through; 
indeed, no pair $(a, b)$ maps to $(1, 0)$ because $d - e \ge 3$.
Note that the Newton polygon of
\[cz^2 + 1 - (1 - z)^d = 0\]
has an edge from $(1, 0)$ to $(2, -1)$
(of slope $-1$ and length 1)
so that 1 has at most one nonzero preimage in $L$.
\end{rmk}

\begin{rmk}
Example \ref{eg:3} similarly extends to the case $e = 2$ and $d = 4$, with the further modified conclusion that 
\[
    \PrePer(f, L)
    = f^{-3}(\{0, 1, \infty\}).
\]
This follows from the fact that now, in the $ab$-plane, 
\[
    (0, 0) \to (-1, -1) \to (1, 0).
\]
\end{rmk}

We conclude this section with a summary result explaining how each of these families has a denominator lemma.

\begin{propn} \label{propn:summary}
Let $K$ be a number field.
Each of the following families $f_t$ 
has the property that if $\PrePer(f_t, K) \ne \PrePer(f, K(T))_t$
then $v(t) \ne -1$ 
for all finite places $v$ outside of the corresponding finite set $S \subseteq M_K$, 
defined below.
\begin{enumerate}[1.]
    \item Let $d \ge 2$ 
    and let \[f(z) = a_d z^d + \ldots + a_0 + T\]
    where $a_i \in K$ and $a_d \ne 0$; 
    set 
    \[S^c = \{v \in M_K : v(a_i) \ge 0 \text{ for all $i$ and } v(a_d) = 0\}.\]
    \item Let $d \ge 2$ and $1 \le e \le d$ 
    and let \[f(z) = \frac{T}{a_e z^e + \ldots + a_d z^d}\]
    where $a_i \in K$ and $a_d, a_e \ne 0$;
    set 
    \[S^c = \{v \in M_K : v(a_i) \ge 0 \text{ for all $i$ and } v(a_d) = v(a_e) = 0\}.\]
    \item Let $d \ge e+2 \ge 5$
    and let \[f(z) = \frac{Tz^e + 1}{(1 - z)^d};\]
    set 
    \[S^c = \{v \in M_K : v(d) = 0\}.\]
\end{enumerate}
\end{propn}

\begin{proof}
Each instance is proved by a double application of the corresponding Example. 
We show this for the first family; the others are analogous.
First, take $L = K(T)$ with $v = \ord_\infty$ and $c = T$.
Since $v(c) = -1$ and $v(a_i) = 0$ for all $i$, 
Example \ref{eg:1} yields
\[\PrePer(f, K(T)) = \{\infty\}.\]
Second, take $L = K$ with $v \not \in S$ and $c = t$.
If $v(t) = -1$ then by Example \ref{eg:1} again
\[\PrePer(f_t, K) = \{\infty\}. \qedhere\]
\end{proof}

\begin{rmk} \label{rmk:berkovich}
All of the Examples in this section may be recast in the language of Berkovich space, as follows. 
Any rational map $f$ may be ``tropicalized'' relative to a finite subtree $Y$ of $\P^1_\an$ by post-composing it with the canonical retraction $r : \P^1_\an \to Y$.
This yields a piecewise linear map $\tilde{f} : Y \to Y$
with the property that $r(f(\zeta)) = \tilde{f}(r(\zeta))$
for all $\zeta$ in some closed connected Berkovich affinoid depending on $f$ and $Y$.
The process of extracting information about the dynamics of $f$ from the dynamics of $\tilde{f}$ may also be completely formalized, 
leading to a formula for the set of $K$-rational preperiodic points of $f$ in terms of various data associated to $\tilde{f}$.
This approach is elaborated in the author's PhD thesis \cite[\S\S 5.3--5.6]{thesis}.
\end{rmk}

\section{Postlude: The quadratic family} \label{sec:quadratic} 

This section is devoted to the most intensely studied one-parameter family:
\[f_t(z) = z^2 + t.\]
We illustrate how our definitions and results apply to this family, and then we give some refined estimates on the total number of its rational preperiodic points.

Here, $d = 2$ and $d' = 1$.
An integral lift is given by $f = (X^2 + TY^2 : Y^2)$.
Hence, the resultant polynomial is 
\[
	R_f(T) = \Res(X^2 + TY^2, Y^2) = 
\begin{pmatrix}
1 &   & 0 &   \\
0 & 1 & 0 & 0 \\
T & 0 & 1 & 0 \\
  & T &   & 1 \\
\end{pmatrix}
= 1.\]
Thus, the domain of definition is $U = \A^1$.
Note that since $\deg R_f = 0 < 4 = 2dd'$, 
we must have $\deg f_\infty < 2$; 
indeed, $f_\infty(z) = \infty$ is constant.

Since $\ord_\infty(T) = -1$, 
Example \ref{eg:1} 
implies $\Gamma = \PrePer(f_t, \Q(T)) = \{\infty\}$.
Thus the injectivity locus is automatically maximal: 
$U^\circ = U$.
As for the surjectivity locus, 
since $f_t^{-1}(\infty) = \{\infty\}$ for all $t$, 
there are never any new branches, only new cycles. 

Now write $t = a/b$ in lowest terms.
Then 
\[H(f_t) 
= \max\{|a|, |b|\}
\quad\text{and}\quad 
v(\Res f_t) = 4 \max\{0, v(b)\}\]
so that the resultant ideal of $f_t$ is $\Res f_t = b^4 \Z$.
In particular, $\Nm \Res f_t = b^4$ and $\omega(\Nm \Res f_t) = \omega(b)$.

By Proposition \ref{propn:summary}(a),
this family has a denominator lemma.
Therefore, 
the total number of its rational preperiodic points is 
\[
    \mathcal{A}(X) = \! \sum_{H(t) \le X} \#{\PrePer}(z^2 + t, \Q) = 1 \cdot \Num(\Q, X) + \mathcal{R}(X)
\]
where, by Theorems \ref{thm:avg} and \ref{thm:den_gives_conj}, 
\begin{equation} \label{eq:estimate}
    \mathcal{R}(X) = \! \sum_{H(t) \le X} \big({\#}{\PrePer}(z^2 + t, \Q) - 1\big) \ll X^{3/2} \log X.
\end{equation}
The estimate \eqref{eq:estimate} is precisely the one obtained by Le Boudec and Mavraki \cite[proof of Theorem 1.1]{LeBoudecMavraki}.
We record the following marginal improvement,
and also show that $\mathcal{R}(X)$ grows at least linearly with $X$.

\begin{thm} \label{thm:HR} \label{thm:gold}
For the quadratic family $f_t(z) = z^2 + t$ over $\Q$, 
\[
C X + O(\sqrt{X} \log X) \le \mathcal{R}(X) \ll X^{3/2} \log \log X \log \log \log X
\]
where
\[
	C = \frac{12}{\pi^2} \left(\frac{2 \pi}{3 \sqrt{3}} + \frac{1 + \sqrt{5}}{2} + 2 \log \frac{1 + \sqrt{5}}{2} \right).
\]
Moreover, if no $f_t$ has a $\Q$-rational cycle of length exceeding 3 (``Poonen's conjecture''), then the lower bound is an equality.
\end{thm}

\begin{rmk} \label{rmk:LBM}
Since every element $a_2 z^2 + 1$ of $\Poly_2^*$ 
is conjugate to $z^2 + a_2$ via $z \mapsto z/a_2$,
Theorem \ref{thm:gold} entails that Poonen's conjecture implies the $d = 2$ case of Le Boudec and Mavraki's conjecture \cite[(1.3)]{LeBoudecMavraki} 
plus an error term:
\[
    \frac{\mathcal{R}(X)}{\Num(\Poly_2^*(\Q), X)}
    \sim \frac{\pi^2}{12} \cdot \frac{C X + O(\sqrt{X}\log X)}{X^2}
    = \frac{C \pi^2}{12 X} + O\left(\frac{\log X}{X^{3/2}}\right).
\]
\end{rmk}

\begin{proof}

\noindent\textit{Upper bound.}
Let $t = a/b$ in lowest terms with $b > 0$.
By Walde \& Russo \cite[Corollary 4]{WaldeRusso},
if $\Gamma_t \ne \{\infty\}$ then $b = c^2$ is a perfect square.
Since $\Gamma_0 = \{0, 1, -1, \infty\}$, 
we have 
\[
    \mathcal{R}(X) = 3 + 2 \sum_{c \le \sqrt{X}} \sum_{\substack{a \le X \\ \gcd(a, c) = 1}} ({\#}\Gamma_{a/c^2} - 1)
\]
for $X \ge 1$.
Invoking Benedetto's bound (cf.~Proposition \ref{propn:benetronc})
\[
    \#\Gamma_{a/c^2} - 1 \ll \omega(c) \log \omega(c)
\]
yields 
\[
    \mathcal{R}(X) \ll \sum_{c \le \sqrt{X}} \omega(c) \log \omega(c) \!\!\! \sum_{\substack{a \le X \\ \gcd(a, c) = 1}} 1.
\]
A simple exercise shows that the inner sum is $\ll \displaystyle \frac{\varphi(c)}{c} \cdot X$
where $\varphi$ is Euler's totient function.
Therefore, 
\begin{equation} \label{eq:bestRX}
\mathcal{R}(X) \ll X \cdot \!\! \sum_{c \le \sqrt{X}} \! \omega(c) \log {\omega(c)} \frac{\varphi(c)}{c}.    
\end{equation}
We claim that the sum on the r.h.s.~of \eqref{eq:bestRX} 
is $O(\sqrt{X} \log \log X \log \log \log X)$.
The key input is the Hardy--Ramanujan inequality 
\cite[Lemma B]{HardyRamanujan}, 
which furnishes absolute constants $L$ and $D$ such that 
\[
\#\{c \le x : \omega(c) = k\} < \frac{Lx}{\log x} \cdot \frac{(\log \log x + D)^{k-1}}{(k-1)!}.
\]
for all $k \ge 1$ and $x \ge 2$.
This and the trivial estimate $\varphi(c) \le c$ imply 
\begin{align*}
\sum_{c \le \sqrt{X}} \omega(c) \log \omega(c) \frac{\varphi(c)}{c} 
& < \frac{L' \sqrt{X}}{\log X} \sum_{k=1}^\infty k \log k \cdot \frac{(\log \log X + D')^{k-1}}{(k-1)!}
\end{align*}
where $L' = 2L$ and $D' = D - \log 2$. 
The desired upper bound follows on noting that for any $\alpha > 0$, 
\[
    \sum_{n=0}^\infty (n + \alpha) \log (n + \alpha) \frac{x^n}{n!} \sim e^x x \log x \qquad (x \to \infty).
\]

\smallskip 

\noindent\textit{Lower bound.}
We establish the lower bound in several steps.

\medskip
\noindent \textit{Step 1: Poonen's theorem.}
Define the following seven subsets of $\Q$:
\begin{align*}
E_0 &= \{0, -1, -2, 1/4, -29/16\}				\\
E_1 &= \{1/4 - \rho^2 : \rho \ne 0, 1/2, -1/2, 3/2, -3/2\}		\\
E_2 &= \{-3/4 - \sigma^2 : \sigma \ne 0, 1/2, -1/2\}				\\
E_3 &= \big\{{-}\frac{\tau^6 + 2\tau^5 + 4\tau^4 + 8\tau^3 + 9\tau^2 + 4\tau + 1}{4\tau^2(\tau + 1)^2} : \tau \ne 0, -1, 1, -2, -1/2\big\} \\
E_4 &= \big\{{-}\frac{2(\eta^2 + 1)}{(\eta^2 - 1)^2} : \eta \ne -1, 0, 1 \big\} 			\\
E_5 &= \big\{{-}\frac{\nu^4 + 2\nu^3 + 2\nu^2 - 2\nu + 1}{(\nu^2 - 1)^2} : \nu \ne -1, 0, 1\big\} 	\\
E_6 &= \big\{{-}\frac{3\mu^4 + 10\mu^2 + 3}{4(\mu^2 - 1)^2} : \mu \ne -1, 0, 1\big\}.
\end{align*}

By Poonen's classification theorem \cite{Poonen}, 
$t \in E_0$ iff $\Gamma_t$ admits (and is conjecturally isomorphic to) one of five exceptional graphs,
while $E_1$--$E_6$ parametrize the six infinite families.
These are depicted in Figures \ref{fig:E0} and \ref{fig:E}.

\begin{figure}
\centering
\begin{tikzpicture}[scale=0.9]\footnotesize
\tikzset{vertex/.style = {shape=circle,fill,minimum size=5,inner sep=0, outer sep=1}}
\tikzset{edge/.style = {->,> = stealth'}}

\draw[black] (-2,1) rectangle (4,-1);
\draw[black] (-2,-1) rectangle (4,-3);
\draw[black] (-2,-3) rectangle (4,-5);
\draw[black] (-2,-5) rectangle (4,-7);
\draw[black] (-2,-7) rectangle (4,-11.25);
\node[] () at (4,-11.25) {};

\node[label=below right:{\normalsize $t=0$}] at (-2,1) {};

\node[vertex,label=below:{$0$}] (a0) at (0,0) {};
\node[vertex,label=below:{$-1$}] (b0) at (1,0) {};
\node[vertex,label=below:{$1$}] (c0) at (2,0) {};
\draw[edge] (a0) to[distance=13,out=120,in=60] (a0);
\draw[edge] (b0) to (c0);
\draw[edge] (c0) to[distance=13,out=120,in=60] (c0);

\node[label=below right:{\normalsize $t=-1$}] at (-2,-1) {};

\node[vertex,label=below:{$-1$}] (a1) at (0,-2) {};
\node[vertex,label=below:{$0$}] (b1) at (1,-2) {};
\node[vertex,label=below:{$1$}] (c1) at (2,-2) {};
\draw[edge] (a1) to[bend left] (b1);
\draw[edge] (b1) to[bend left] (a1);
\draw[edge] (c1) to (b1);

\node[label=below right:{\normalsize $t=-2$}] at (-2,-3) {};
\node[vertex,label=below:{$0$}]  (a2) at (-1,-4) {};
\node[vertex,label=below:{$-2$}] (b2) at (0,-4) {};
\node[vertex,label=below:{$2$}]  (c2) at (1,-4) {};
\node[vertex,label=below:{$1$}]  (d2) at (2,-4) {};
\node[vertex,label=below:{$-1$}] (e2) at (3,-4) {};
\draw[edge] (a2) to (b2);
\draw[edge] (b2) to (c2);
\draw[edge] (c2) to[distance=13,out=120,in=60] (c2);
\draw[edge] (d2) to (e2);
\draw[edge] (e2) to[distance=13,out=120,in=60] (e2);

\node[label=below right:{\normalsize $t=1/4$}] at (-2,-5) {};

\node[vertex,label=below:{$-\frac12$}] (a3) at (0.5,-6) {};
\node[vertex,label=below:{$\frac12$}] (b3) at (1.5,-6) {};
\draw[edge] (a3) to (b3);
\draw[edge] (b3) to[distance=13,out=120,in=60] (b3);

\node[label=below right:{\normalsize $t=-29/16$}] at (-2,-7) {};

\node[vertex,label=left:{$\frac34$}] (a) at (-0.5,-8.5) {};
\node[vertex,label=above:{$-\frac54$}] (b) at (0.5,-8.5) {};
\node[vertex,label=above:{$-\frac14$}] (c) at (1.5,-8.5) {};
\node[vertex,label=right:{$\frac14$}] (d) at (2.5,-8.5) {};
\node[vertex,label=left:{$-\frac34$}] (e) at (0,{-8.5-sqrt(3)/2}) {};
\node[vertex,label=left:{$\frac54$}] (f) at (1,{-8.5-sqrt(3)/2}) {};
\node[vertex,label=below:{$-\frac74$}] (g) at (2,{-8.5-sqrt(3)/2}) {};
\node[vertex,label=below:{$\frac74$}] (h) at (1.5,{-8.5-sqrt(3)}) {};
\draw[edge] (a) to (b);
\draw[edge] (b) to[bend left] (c);
\draw[edge] (c) to[bend left] (g);
\draw[edge] (g) to[bend left] (f);
\draw[edge] (f) to[bend left] (c);
\draw[edge] (e) to (b);
\draw[edge] (d) to[bend left] (g);
\draw[edge] (h) to[bend left] (f);
\end{tikzpicture}
\caption{The five exceptional portraits.} \label{fig:E0}
\end{figure}

\begin{figure}
\centering 
\makebox[\textwidth][c]{
\begin{tikzpicture}[scale=0.9]
\tikzset{vertex/.style = {shape=circle,fill,minimum size=5,inner sep=0, outer sep=1}}
\tikzset{edge/.style = {->,> = stealth'}}

\node[vertex] (a) at (0,0) {};
\node[vertex] (b) at (1,0) {};
\node[vertex] (c) at (2,0) {};
\node[vertex] (d) at (3,0) {};
\draw[edge] (b) to (a);
\draw[edge] (a) to[distance=13,out=150,in=210] (a);
\draw[edge] (c) to (d);
\draw[edge] (d) to[distance=13,out=30,in=-30] (d);

\node[vertex] (a2) at (6,0) {};
\node[vertex] (b2) at (7,0) {};
\node[vertex] (c2) at (8,0) {};
\node[vertex] (d2) at (9,0) {};
\draw[edge] (a2) to (b2);
\draw[edge] (b2) to[bend left] (c2);
\draw[edge] (c2) to[bend left] (b2);
\draw[edge] (d2) to (c2);

\node[vertex] (a3) at (3.5,{3.5-sqrt(3)/2}) {};
\node[vertex] (b3) at (4.5,{3.5-sqrt(3)/2}) {};
\node[vertex] (c3) at (5.5,{3.5-sqrt(3)/2}) {};
\node[vertex] (d3) at (4,   3.5) {};
\node[vertex] (e3) at (5,   3.5) {};
\node[vertex] (f3) at (4.5,{3.5+sqrt(3)/2}) {};
\draw[edge] (a3) to[bend left] (d3);
\draw[edge] (d3) to[bend left] (e3);
\draw[edge] (e3) to[bend left] (b3);
\draw[edge] (b3) to[bend left] (d3);
\draw[edge] (f3) to[bend left] (e3);
\draw[edge] (c3) to[bend left] (b3);

\node[vertex] (a4) at (-3,-3.5) {};
\node[vertex] (b4) at (-2,-3.5) {};
\node[vertex] (c4) at (-1,-3.5) {};
\node[vertex] (d4) at ( 0,-3.5) {};
\node[vertex] (e4) at (-1.5,{-3.5+sqrt(3)/2}) {};
\node[vertex] (f4) at (-1.5,{-3.5-sqrt(3)/2}) {};
\draw[edge] (b4) to (a4);
\draw[edge] (a4) to[distance=13,out=150,in=210] (a4);
\draw[edge] (c4) to (d4);
\draw[edge] (d4) to[distance=13,out=30,in=-30] (d4);
\draw[edge] (e4) to (c4);
\draw[edge] (f4) to (c4);

\node[vertex] (a5) at (8.5,-3.5) {};
\node[vertex] (b5) at (9.5,-3.5) {};
\node[vertex] (c5) at (10.5,-3.5) {};
\node[vertex] (d5) at (11.5,-3.5) {};
\node[vertex] (e5) at (12.5,{-3.5+sqrt(3)/2}) {};
\node[vertex] (f5) at (12.5,{-3.5-sqrt(3)/2}) {};
\draw[edge] (a5) to (b5);
\draw[edge] (b5) to[bend left] (c5);
\draw[edge] (c5) to[bend left] (b5);
\draw[edge] (d5) to (c5);
\draw[edge] (e5) to (d5);
\draw[edge] (f5) to (d5);

\node[vertex] (a6) at (3,-4) {};
\node[vertex] (b6) at (4,-4) {};
\node[vertex] (c6) at (5,-4) {};
\node[vertex] (d6) at (6,-4) {};
\node[vertex] (a7) at (3,-3) {};
\node[vertex] (b7) at (4,-3) {};
\node[vertex] (c7) at (5,-3) {};
\node[vertex] (d7) at (6,-3) {};
\draw[edge] (b6) to (a6);
\draw[edge] (a6) to[distance=13,out=150,in=210] (a6);
\draw[edge] (c6) to (d6);
\draw[edge] (d6) to[distance=13,out=30,in=-30] (d6);
\draw[edge] (a7) to (b7);
\draw[edge] (b7) to[bend left] (c7);
\draw[edge] (c7) to[bend left] (b7);
\draw[edge] (d7) to (c7);

\draw[black] (-1, 1) rectangle (4,-1);
\draw[black] (5, 1) rectangle (10,-1);
\draw[black] (3, 5) rectangle (6,2);

\draw[black] (-4, -2) rectangle (1, -5);
\draw[black] (2, -2) rectangle (7, -5);
\draw[black] (8, -2) rectangle (13, -5);

\draw[gray] (-1, -2) -- (0, -1);
\draw[gray] (4,  -2) -- (3, -1);
\draw[gray] (5,  -2) -- (6, -1);
\draw[gray] (10, -2) -- (9, -1);

\node[label=below right:{\normalsize $E_1$}] at (-1,1) {};
\node[label=below right:{\normalsize $E_2$}] at (5,1) {};
\node[label=below right:{\normalsize $E_3$}] at (3,5) {};
\node[label=below right:{\normalsize $E_4$}] at (-4,-2) {};
\node[label=below right:{\normalsize $E_6$}] at (2,-2) {};
\node[label=below right:{\normalsize $E_5$}] at (8,-2) {};
\end{tikzpicture}
}
\caption{The six infinite families. The gray lines indicate inclusions.} \label{fig:E}
\end{figure}

Note that there are more restrictions on the parameters $\rho, \sigma, \tau, \eta, \nu, \mu \in \Q$ than those explicitly appearing in Poonen's paper; 
we have filled these in to ensure that $E_0 \cap (E_1 \cup \ldots \cup E_6) = \varnothing$.
The situation is summarized in Table \ref{tab:poonen}.

\begin{table}
\vspace{1em}
\renewcommand{\arraystretch}{1.2}
\begin{tabular}{|c|l|l|} \hline 
Theorem & excludes & corresponding $t$ \\ \hline \hline 
1.1 & \phantom{[}$\rho = 0$                      & [$1/4$] \\ \hline 
1.2 & \phantom{[}$\sigma = 0$                    & [$-3/4$] \\ \hline 
1.3 & \phantom{[}$\tau = 0, -1$                  & [$\infty$] \\ \hline 
2.1 & \phantom{[}$\eta = 0, \pm 1$               & [$-3/4, \infty$] \\ \hline 
\multirow{3}{*}{3.1} 
 & [$\sigma = \pm 1/2$]    & \multirow{2}{*}{$-1$} \\  
 & [$\nu = 0$] & \\ \cline{2-3}
 & [$\rho = \pm 1/2$]                 & $0$ \\  \hline 
\multirow{2}{*}{3.2} 
 & \phantom{[}$\eta = 0$   & \multirow{2}{*}{[$-2$]} \\ 
 & [$\rho = \pm 3/2$]      & \\  \hline 
3.3 & \phantom{[}$\nu = \pm 1, 0$                & [$\infty, -1$] \\ \hline 
3.4 & [$\tau = -2, -1/2, 1$]          & $-29/16$ \\ \hline 
\end{tabular}
\vspace{1em}
\caption{Restrictions on Poonen's parameters.} \label{tab:poonen}
\end{table}

It follows from Poonen's Theorems 2.1, 2.2, 3.2, 3.3 that
\begin{enumerate}[i.]
	\item $E_1 \cap E_2 = E_6$ 
	\item $(E_1 \cup E_2) \cap E_3 = \varnothing$ 
	\item $E_4 \subseteq E_1$ 
	\item $E_5 \subseteq E_2$ 
\end{enumerate}
respectively.
For the nonce, let $n = n_X$ and $m = m_X$ be the discrete measures on $\Q$ defined by 
\[
    n(A) = \Num(A, X)
\]
and 
\[
	m(A) = \!\! \sum_{\substack{t \in A \\ H(t) \le X}} \!\! \big({\#}{\PrePer}(z^2 + t, \Q) - 1\big).
\]
By definition, the total mass of $m$ is $m(\Q) = \mathcal{R}(X)$.
By i, iii, and iv, the support of $m$ satisfies $\operatorname{supp}(m) = E \supseteq E_* := E_0 \cup E_1 \cup E_2 \cup E_3$.
Hence, 
\[
	\Num(E, X) \ge n(E_*) \quad\text{and}\quad \mathcal{R}(X) \ge m(E_*)
\]
with equality if and only if Poonen's conjecture is true.
The rest of the proof is concerned with evaluating $m(E_*)$ and $n(E_*)$ unconditionally.
By i and ii,
\begin{equation} \label{eq:n_meas}
    n(E_*) = 5 + n(E_1) + n(E_2) - n(E_6) + n(E_3)
\end{equation}
for all $X \ge 29$. 
Similarly, 
\[
	m(E_*) = m(E_0) + m(E_1) + m(E_2) - m(E_6) + m(E_3).
\]
Now, 
\[
	m(E_0) 
	= 3 
	+ 3 
	+ 5 
	+ 2 
	+ 8 
        = 21
\]
for all $X \ge 29$;
trivially, 
\[
	m(E_3) = 6 n(E_3) \quad\text{and}\quad m(E_6) = 8 n(E_6);
\]
and finally, the inclusions in i, iii, and iv imply 
\[
	m(E_1) = 4n(E_1) + 2n(E_4) + 4n(E_6)
		\quad\text{and}\quad 
	m(E_2) = 4n(E_2) + 2n(E_5) + 4n(E_6)
\]
because the graph acquires 2 points if $t \in E_4$ (resp.~$E_5$) or 4 points if $t \in E_6$.
Combining the preceding four equation displays gives 
\begin{equation} \label{eq:8.5}
	m(E_*) = 21 + 4\big( n(E_1) + n(E_2) \big) + 6 n(E_3) + 2\big( n(E_4) + n(E_5) \big)
\end{equation}
for all $X \ge 29$.

\medskip 

\noindent \textit{Step 2: Heights.}
To estimate 
\[n(E_i) = \Num(E_i, X) \qquad (i = 1, \ldots, 6)\]
we use the fact that $E_i = \phi_i(U_i)$ for some explicit rational function $\phi_i$ and subset $U_i$ of $\Q$.
By basic properties of heights, 
if $e = \deg \phi$, 
then there exist constants $C \ge c > 0$ such that 
\[
	cH(t)^e \le H(\phi(t)) \le CH(t)^e
\]
for all $t$ in $\Q$.
In particular,
\[
	\Num(U, (X/C)^{1/e}) \le \Num(\phi(U), X) \le \Num(U, (X/c)^{1/e})
\]
for any $U \subseteq \Q$.
Since 
\[
	\Num(\P^1(\Q), X) = k X^2 + O(X \log X), \qquad k = \frac{12}{\pi^2}
\]
it follows that if $U$ is cofinite, 
then
\[
	\frac{k}{C^{2/e}} X^{2/e}
 \lesssim \Num(\phi(U), X) \lesssim
    \frac{k}{c^{2/e}} X^{2/e}.
\]

This shows two things: that for $i > 2$ the terms $\Num(E_i, X)$ are $O(\sqrt{X})$;
and that for $i = 1, 2$ 
such basic estimates are inadequate unless $C = c$.
At any rate, 
\eqref{eq:n_meas}
and 
\eqref{eq:8.5} simplify to 
\begin{equation} \label{eq:n_meas2}
    n_X(E_*) = \Num(E_1, X) + \Num(E_2, X) + O(\sqrt{X})
\end{equation}
and 
\begin{equation} \label{eq:8.6}
    m_X(E_*) = 4n_X(E_*) + O(\sqrt{X})
\end{equation}
respectively. Thus it remains to count $E_1$ and $E_2$.

To that end,
let $\phi(t) = \phi_c(t) = c/4 - t^2$ where $c \in \{1, -3\}$.
Then 
\begin{align*}
	\Num(E_1, X) &= \Num(\phi_1(\Q^\times), X) + O(1) \\ 
\shortintertext{and}
	\Num(E_2, X) &= \Num(\phi_{-3}(\Q^\times), X) + O(1).
\end{align*}
Since $\phi$ is even, 
\begin{equation} \label{eq:8.7}
	\Num(\phi(\Q^\times), X) 
	= \#\{(a, b) : a, b > 0, \, \gcd(a, b) = 1, \, H(\phi(a/b)) \le X\}.
\end{equation}
In order to compute the na\"ive height, 
we must express
\[
\phi(a/b) = c/4 - a^2/b^2 = \frac{cb^2 - 4a^2}{4b^2}
\]
in lowest terms.
Some elementary fiddling shows that 
\[
d := \gcd(cb^2 - 4a^2, 4b^2) = 
\begin{cases}
		1	& b \equiv 1 \pmod 2 \\
		16	& b \equiv 2 \pmod 4 \\ 
		4	& b \equiv 0 \pmod 4
	\end{cases}
\]
and it's not hard to see that 
\[
    H(\phi(a/b)) = \max \{\lvert cb^2 - 4a^2 \rvert, 4b^2 \}/d \le X
\]
if and only if 
\[
    b \le \tfrac{1}{2} \sqrt{dX}
\quad\text{and}\quad 
    a \le \tfrac{1}{2}\sqrt{dX + cb^2}.
\]
Note that these claims rely on $c$ being 1 mod 4 and less than 4.

Thus, the region on the r.h.s.~of \eqref{eq:8.7}
decomposes---according to the residue class of $b$---into parts of the form 
\[
    S_{q,r} := \{(a, b) : 1 \le b \le \tfrac{1}{2}\sqrt{dX}, \,
    b \equiv r \, (q), \,
    1 \le a \le \tfrac{1}{2}\sqrt{dX + cb^2}, \,
    \gcd(a, b) = 1\}
\]
where $(q, r) \in \{(2, 1), (4, 2), (4, 0)\}$ and $d$ is 1, 16, or 4 accordingly (illustrated in Figures \ref{fig:3} and \ref{fig:1} in \textcolor{my_red}{\textbf{red} $\bullet$}, 
\textcolor{my_green}{\textbf{green} $\circ$}, and 
\textcolor{my_blue}{\textbf{blue} $\times$} respectively).

\begin{figure}
\centering
\begin{tikzpicture}[smooth,scale=0.33]
\draw[<->] (-2,0)--(24,0);
\draw[<->] (0,-2)--(0,22);
\draw (0,0) node[below left]{$0$};
\foreach \i in {5,10,15,20} {
\draw (\i,.1)--(\i,-.1) node[below]{$\i$};
\draw (.1,\i)--(-.1,\i) node[left]{$\i$};
}

\draw[my_red,thick] (0,5)--(2.5,5);
\draw[my_blue,thick,dashed] (0,10)--(5,10);
\draw[my_green,line width=1.5,loosely dotted] (0,20)--(10,20);

\draw[my_red,thick,domain=2.5:5] plot(\x,{2*sqrt((25-\x*\x)/3)});
\draw[my_blue,thick,dashed,domain=5:10] plot(\x,{2*sqrt((100-\x*\x)/3)});
\draw[my_green,line width=1.5,loosely dotted,domain=10:20] plot(\x,{2*sqrt((400-\x*\x)/3)});

\foreach \point in {(1,1), (1,3), (1,5),
(2,1), (2,3), (2,5),
(3,1),
(4,1), (4,3)} 
\fill[my_red] \point circle (.17);

\foreach \point in {(1,4), (1,8), 
(3,4), (3,8),
(5,4), (5,8),
(7,4), (7,8),
(9,4)}
{
	\draw[my_blue,line width=1.3] \point--++(45:.2);
	\draw[my_blue,line width=1.3] \point--++(135:.2);
	\draw[my_blue,line width=1.3] \point--++(225:.2);
	\draw[my_blue,line width=1.3] \point--++(315:.2);
}

\foreach \point in {(1,2), (1,6), (1,10), (1,14), (1,18),
(3,2), (3,10), (3,14),
(5,2), (5,6), (5,14), (5,18),
(7,2), (7,6), (7,10), (7,18),
(9,2), (9,10), (9,14),
(11,2), (11,6), (11,10), (11,14), (11,18),
(13,2), (13,6), (13,10), (13,14),
(15,2), (15,14),
(17,2), (17,6), (17,10),
(19,2), (19,6)}
\draw[my_green,line width=.9] \point circle (.15);
\end{tikzpicture}
\caption{$\Num(\phi_{-3}(\Q^\times), 100) = \textcolor{my_red}{9} + \textcolor{my_green}{35} + \textcolor{my_blue}{9} = 53$} \label{fig:3}
\end{figure}

\begin{figure}
\centering
\begin{tikzpicture}[smooth,scale=0.33]
\draw[<->] (-2,0)--(24,0);
\draw[<->] (0,-2)--(0,22);
\draw (0,0) node[below left]{$0$};
\foreach \i in {5,10,15,20} {
\draw (\i,.1)--(\i,-.1) node[below]{$\i$};
\draw (.1,\i)--(-.1,\i) node[left]{$\i$};
}

\draw[my_red,thick] (0,5)--({5*sqrt(5)/2},5);
\draw[my_blue,thick,dashed] (0,10)--({5*sqrt(5)},10);
\draw[my_green,line width=1.5,loosely dotted] (0,20)--({10*sqrt(5)},20);

\draw[my_red,thick,domain=5:5*sqrt(5)/2] plot(\x,{2*sqrt(\x*\x-25)});
\draw[my_blue,thick,dashed,domain=10:5*sqrt(5)] plot(\x,{2*sqrt(\x*\x-100)});
\draw[my_green,line width=1.5,loosely dotted,domain=20:10*sqrt(5)] plot(\x,{2*sqrt(\x*\x-400)});

\foreach \point in {(1,1), (1,3), (1,5),
(2,1), (2,3), (2,5),
(3,1), (3,5),
(4,1), (4,3), (4,5),
(5,1), (5,3)} 
\fill[my_red] \point circle (.17);

\foreach \point in {(1,4), (1,8), 
(3,4), (3,8),
(5,4), (5,8),
(7,4), (7,8),
(9,4), (9,8)}
{
	\draw[my_blue,line width=1.3] \point--++(45:.2);
	\draw[my_blue,line width=1.3] \point--++(135:.2);
	\draw[my_blue,line width=1.3] \point--++(225:.2);
	\draw[my_blue,line width=1.3] \point--++(315:.2);
}

\foreach \point in {(1,2), (1,6), (1,10), (1,14), (1,18),
(3,2), (3,10), (3,14),
(5,2), (5,6), (5,14), (5,18),
(7,2), (7,6), (7,10), (7,18),
(9,2), (9,10), (9,14),
(11,2), (11,6), (11,10), (11,14), (11,18),
(13,2), (13,6), (13,10), (13,14), (13,18),
(15,2), (15,14),
(17,2), (17,6), (17,10), (17,14), (17,18),
(19,2), (19,6), (19,10), (19,14), (19,18)}
\draw[my_green,line width=.9] \point circle (.15);
\end{tikzpicture} 
\caption{$\Num(\phi_1(\Q^\times), 100) = \textcolor{my_red}{13} + \textcolor{my_green}{41} + \textcolor{my_blue}{10} = 64$} \label{fig:1}
\end{figure}

\medskip 

\noindent \textit{Step 3: Analytic number theory.}
Our final task is to count the number of points in $S_{q,r}$.
By Fubini's theorem, 
\begin{equation} \label{eq:8.8}
\# S_{q,r}
= \!\! \sum_{\substack{b \le \frac{1}{2} \! \sqrt{dX} \\ b \equiv r \, (q) }} \! \! F_b\big(\tfrac{1}{2} \sqrt{dX + cb^2}\big)
\end{equation}
where $F_b(Y) := \#\{1 \le a \le Y : \gcd(a, b) = 1\}$.
Appealing to 
\[\sum_{k \mid n} \frac{\mu(k)}{k} = \frac{\varphi(n)}{n}\]
we have 
\begin{align*}
F_b(Y) 
&= \sum_{1 \le a \le Y} \sum_{k \mid \gcd(a, b)} \mu(k) \\
&= \sum_{k \mid b} \mu(k) \sum_{1 \le a \le Y, \, k \mid a} 1 \\
&= \sum_{k \mid b} \mu(k) \Big\lfloor \frac{Y}{k} \Big\rfloor \\ 
&= Y \sum_{k \mid b} \frac{\mu(k)}{k} - \sum_{k \mid b} \mu(k) \Big\{ \frac{Y}{k} \Big\}
\end{align*}
whence 
\[\Big \lvert F_b(Y) - Y\frac{\varphi(b)}{b} \Big \rvert \le \sum_{k | b} \Big \lvert \Big\{ \frac{Y}{k} \Big\} \Big \rvert < 2^{\omega(b)}.\]
because the sum is supported on the set of primes dividing $b$.
It's well known that $2^{\omega(b)} \le \tau(b)$,
the number of divisors of $b$; 
and by Dirichlet's theorem, 
\[\sum_{b \le Y} \tau(b) \ll Y \log Y.\]
Plugging these into \eqref{eq:8.8} yields
\begin{align} \label{eq:8.9}
\# S_{q,r}
 &= \!\! \sum_{\substack{b \le \frac{1}{2} \! \sqrt{dX} \\ b \equiv r \, (q) }} \! \! \Big(\tfrac{1}{2} \sqrt{dX + cb^2} \cdot \frac{\varphi(b)}{b} + O(\tau(b))\Big) \nonumber 
 \\
 &= \tfrac{1}{2} G_{q,r} \big(\tfrac{1}{2}\sqrt{dX}\big) + O\big(\sqrt{X} \log X\big) 
\end{align}
where 
\[
    G_{q,r}(Y) = \sum_{\substack{b \le Y \\ b \equiv r \,(q)}} \sqrt{4Y^2 + cb^2} \cdot \frac{\varphi(b)}{b}.
\]

It remains to estimate $G_{q,r}$.
With an eye to applying Abel summation, 
we first show that there exist constants $C_{q,r} > 0$ such that 
\begin{equation} \label{eq:Hqr}
    H_{q,r}(Y) := \sum_{\substack{b \le Y \\ b \equiv r \, (q)}} \! \frac{\varphi(b)}{b} \\
              = C_{q,r} Y + O(\log Y).
\end{equation}
Indeed, 
\[
H_{q,r}(Y)  = \sum_{\substack{b \le Y \\ b \equiv r \, (q)}} \sum_{k | b} \frac{\mu(k)}{k}
           = \sum_{k \le Y} \frac{\mu(k)}{k} \#\{1 \le b \le Y : k \text{ divides $b$ and } b \equiv r \ (q)\}.\]
The set of $b$'s can be counted
using a generalization of the Chinese Remainder Theorem 
to non-coprime moduli:
the system 
\begin{align*}
x & \equiv r \pmod n \\ 
x & \equiv s \pmod m
\end{align*}
has a unique solution modulo $\operatorname{lcm}(n, m)$ 
if and only if $r \equiv s$ (mod $\gcd(n, m)$).
Thus
\[
\#\{1 \le b \le Y : b \equiv 0 \ (k) \text{ and } b \equiv r \ (q)\}
= [\gcd(k, q)|r] \cdot \Big\lfloor \frac{Y}{\lcm(k, q)} \Big\rfloor\]
where the Iverson bracket $[P]$ is 0 or 1  according as the statement $P$ is true or false.
Writing $a_k = \mu(k) \gcd(k, q) [\gcd(k, q)|r]$
it follows that 
\[
H_{q,r}(Y) = \frac{1}{q} \sum_{k \le Y} \frac{a_k}{k^2} \cdot Y + O(\log Y)
\]
since the error is bounded by the $\lfloor Y \rfloor$\textsuperscript{th} harmonic number.
Now, $|a_k| \le q$ for all $k$, so
\[\sum_{k\le Y} \frac{a_k}{k^2} = \sum_{k=1}^\infty \frac{a_k}{k^2} + O(1/Y)\]
and therefore, 
\[H_{q,r}(Y) = \underbrace{\frac{1}{q} \sum_{k=1}^\infty \frac{a_k}{k^2}}_{C_{q,r}} \mathbin{\cdot} Y + O(\log Y).\]

We shall also require the exact numerical value of the constant $C_{q,r}$ for our three particular pairs $(q, r)$.
The coefficients $a_k$ are multiplicative, 
and a straightforward computation with Euler products supplies the neat formula 
\begin{equation} \label{eq:Cqr}
C_{q, r} 
= \frac{6}{\pi^2}
\cdot \frac{1}{\varphi(q)} \prod_{p \mid q} \frac{p - [p|r]}{p + 1}.
\end{equation}

Without further ado, let $g(t) = \sqrt{4Y^2 + ct^2}$. 
Note that 
\[g(Y) = \sqrt{4 + c} \cdot Y \quad\text{and}\quad g'(t) = \frac{ct}{g(t)}.\]
By Abel summation, 
\[G_{q,r}(Y) = \sqrt{4 + c} \cdot Y H_{q,r}(Y) - \int_1^Y \frac{ct H_{q,r}(t)}{g(t)} \, dt. \]
Using \eqref{eq:Hqr}
we get
\[\sqrt{4 + c} \cdot Y H_{q,r}(Y) = \sqrt{4 + c} \cdot C_{q,r} Y^2 + O(Y \log Y)\]
while 
\[\int_1^Y \frac{ct H_{q,r}(t)}{g(t)} \, dt
= C_{q,r} \int_1^Y \frac{ct^2}{g(t)} \, dt + O\Big(\int_1^Y \frac{t \log(t)}{g(t)} \, dt\Big).\]
Since $Y \ge t \ge 1$ in the range of integration,
\[g(t) = \sqrt{4Y^2 + ct^2} \ge \sqrt{4t^2 + ct^2} = \sqrt{4+c} \cdot t\]
so that
\[\int_1^Y \frac{t \log t}{g(t)} \, dt \ll \int_1^Y \log t \ll Y \log Y.\]
For the main term we use the explicit antiderivative
\[
\int \frac{ct^2}{\sqrt{4Y + ct^2}} \, dt 
= \frac{1}{2} t \sqrt{4Y + ct^2} - \frac{2}{\sqrt{|c|}} Y^2 A\Big(\frac{t\sqrt{|c|}}{2Y}\Big)
\]
where $A(x) = \operatorname{arsinh}(x)$ if $c > 0$ and $A(x) = \arcsin(x)$ if $c < 0$ 
(the point is that $A'(x) = (1 + \sgn(c) x^2)^{-1/2}$).
In either case, $A(x) = x + O(x^3)$ as $x \to 0$.
Since 
\[
\frac{1}{2} g(1) - \frac{2}{\sqrt{|c|}} Y^2 A\Big(\frac{\sqrt{|c|}}{2Y}\Big) 
= O(Y^{-1})
\]
it follows that 
\begin{equation} \label{eq:8.12}
G_{q,r}(Y) = \underbrace{\Big[\frac{\sqrt{4 + c}}{2} + \frac{2}{\sqrt{|c|}}A\Big(\frac{\sqrt{|c|}}{2}\Big)\Big]}_{\gamma(c)} C_{q,r} Y^2 + O(Y \log Y).
\end{equation}

\medskip

\noindent 
\textit{Step 4: Wrap-up.}
Inserting \eqref{eq:8.12} with $Y = \frac{1}{2}\sqrt{dX}$ into \eqref{eq:8.9} yields
\[
\# S_{q,r} 
= \frac{1}{8} \gamma(c) d C_{q,r} X + O(\sqrt{X} \log X).
\]
Summing this over $(q, r) \in \{(2, 1), (4, 2), (4, 0)\}$ gives 
\[
\Num(\phi_c(\Q^\times), X) 
= \frac{1}{8} \gamma(c) \big(C_{2,1} + 16 C_{4, 2} + 4 C_{4, 0}\big) X + O(\sqrt{X} \log X)
\]
by \eqref{eq:8.7}. 
The formula \eqref{eq:Cqr} implies 
\[
    C_{2,1} = \frac{4}{\pi^2} \quad\text{and}\quad C_{4,2} = C_{4,0} = \frac{1}{\pi^2}
\]
so
\[
    \Num(\phi_c(\Q^\times), X)
    = \frac{3}{\pi^2} \gamma(c) X + O(\sqrt{X} \log X).
\]
Plugging this into \eqref{eq:n_meas2}
shows that
\[
n_X(E_*) = \frac{3}{\pi^2} \big(\gamma(1) + \gamma(-3)\big) X + O(\sqrt{X} \log X)
\]
and the lower bound in Theorem \ref{thm:gold} follows on evaluating 
\[
\gamma(1) = \frac{\sqrt{5}}{2} + 2 \log \frac{1 + \sqrt{5}}{2} \quad\text{and}\quad\gamma(-3) = \frac{1}{2} + \frac{2\pi}{3\sqrt{3}}.
\qedhere
\]
\end{proof}

\begin{rmk}
Without more information 
about the distribution of preperiodic points, 
the upper bound in Theorem \ref{thm:HR} probably cannot be improved.
Indeed, one can show (using a basic case of Shiu's theorem \cite[Theorem 1]{Shiu}) that 
\[\sum_{n \le x} \omega(n) \log \omega(n) = x \log \log x \log \log \log x + O(x \log\log\log\log x)\]
and heuristically one would expect
\[\sum_{n \le x} \omega(n) \log \omega(n) \frac{\varphi(n)}{n} \asymp \sum_{n \le x} \omega(n) \log \omega(n).\]
\end{rmk}

\begin{rmk} \label{rmk:cv}
    Every degree-2 rational map with a critical point of period 2 is either totally ramified
    or else conjugate to a member of the family
    \[
        f_t(z) = \frac{t}{z^2 - z}
    \]
    via a M\"obius map 
    moving the critical point, its image, and its non-periodic preimage
    to $\infty$, 0, and 1 respectively.
    By Example \ref{eg:2}, 
    the average number of preperiodic points of this family is 3.

    The classification theorem of Canci and Vishkautsan \cite{CV} 
    implies that
    \[
    \mathcal{R}(X) \ge 2 \cdot \Num(\psi(\Q), X) + O(X^{2/3}), \qquad \psi(r) = r^2 - r,
    \]
    with equality if no $f_t$ has a cycle longer than 2.
    (\textit{Sketch.} The graphs in rows 2, 3, 5, 6, 7 of \cite[Table 5.2]{CV}
    are parametrized by rational curves of degrees 
    3, 2, 4, 6, 6 respectively.)
    By a computation completely analogous to the one carried out in the proof of Theorem \ref{thm:gold}, and noting that $\psi(r) = \psi(1 - r)$,
    it can be shown (cf.~Figure \ref{fig:CV}) that 
    \begin{equation} \label{eq:CV_asymp}
        \Num(\psi(\Q), X) = \frac{3}{\pi^2}\left(\frac{\sqrt{5}}{2} + 2 \log \frac{1 + \sqrt{5}}{2}\right) X + O(\sqrt{X} \log X).
    \end{equation}
\end{rmk}

\begin{figure}[h]
\centering 
\begin{tikzpicture}[smooth,scale=0.33]
\draw[<->] (-2,0)--(18,0);
\draw[<->] (0,-2)--(0,12);
\draw (0,0) node[below left]{$0$};
\foreach \i in {5,10,15} 
\draw (\i,.1)--(\i,-.1) node[below]{$\i$};
\foreach \j in {5,10} 
\draw (.1,\j)--(-.1,\j) node[left]{$\j$};

\draw[my_purple,thick] (0,0)--(5,10)--({5+5*sqrt(5)},10);
\draw[my_purple,domain=10:5+5*sqrt(5),thick] plot(\x,\x-100/\x);

\foreach \point in {(1,1), (1,2), 
(2,1), (2,3),
(3,1), (3,2), (3,4), (3,5),
(4,1), (4,3), (4,5), (4,7),
(5,1), (5,2), (5,3), (5,4), (5,6), (5,7), (5,8), (5,9),
(6,1), (6,5), (6,7),
(7,1), (7,2), (7,3), (7,4), (7,5), (7,6), (7,8), (7,9), (7,10),
(8,1), (8,3), (8,5), (8,7), (8,9),
(9,1), (9,2), (9,4), (9,5), (9,7), (9,8), (9,10),
(10,1), (10,3), (10,7), (10,9),
(11,2), (11,3), (11,4), (11,5), (11,6), (11,7), (11,8), (11,9), (11,10),
(12,5), (12,7),
(13,6), (13,7), (13,8), (13,9), (13,10),
(14,9)} 
\fill[my_purple] \point circle (.17);
\end{tikzpicture}
\caption{$\Num(\psi(\Q), 100) = \textcolor{my_purple}{65}$} \label{fig:CV}
\end{figure}

\begin{rmk} \label{rmk:geometry}
The form of the asymptotics obtained
(in particular, the shapes of the constants) 
is consistent with the following ``geometry of numbers'' heuristics.

    Let $\phi \in \Rat_e(\Q)$ and pick a lift $f, g$ defined over $\Z$.
    Since $\phi$ covers its own image,
\[
    \Num(\phi(\P^1(\Q)), X)
        = \!\! 
    \sum_{\substack{P \in \P^1(\Q) \\ H(\phi(P)) \le X}} 
    \! \frac{1}{\#\big(\phi^{-1}(\phi(P)) \cap \P^1(\Q)\big)}.
\]
Put 
$j = \#{\big(\phi^{-1}(\phi(T : 1)) \cap \P^1(\Q(T))\big)}$.
By Hilbert's irreducibility theorem,
the set 
\[
\{P \in \P^1(\Q) : \#\big(\phi^{-1}(\phi(P)) \cap \P^1(\Q)\big) \ne j\}
\]
is thin.
From this and standard height estimates it follows that 
\begin{equation} \label{eq:geometry1}
\Num(\phi(\P^1(\Q)), X) 
= \frac{1}{j} \#\{P \in \P^1(\Q) : H(\phi(P)) \le X\} + O(X^{1/e}).
\end{equation}
Note that if $e = 2$ then the error term is actually $O(1)$.
The main term 
    is equal to \\ (*)
    \[
    \frac{1}{2j} \sum_d \# \{(a, b) \in \Z^2_\vis : \gcd(f(a, b), g(a, b)) = d, \, \max\{|f(a, b)|, |g(a, b)|\} \le dX\}
    \]
    where $\Z_\vis = \{(a, b) \in \Z^2 : \gcd(a, b) = 1\}$ is the set of lattice points ``visible'' from the origin, and the factor of a half accounts for the symmetry $(a : b) = (-a : -b)$
    coming from the roots of unity in $\Q$.
    Note that for any $r > 0$, the set
    \[
        \{(x, y) \in \R^2 : \max\{|f(x, y)|, |g(x, y)|\} \le r\}
    \]
    is precisely the preimage under $\Phi = (f, g)$ of the $\ell_\infty$-ball $B(0, r)$. 
    By homogeneity, 
    $\Phi^{-1}(B(0, r)) = r^{1/e} \Phi^{-1}(B(0, 1))$.
    Let $D_\Phi = \Phi^{-1}(B(0, 1))$ 
    and let $\vol(D_\Phi)$ denote its area.
    Then a basic principle of the geometry of numbers 
    (see, e.g.,~\cite[Chapter 3, Theorem 5.1]{Lang}) 
    implies that
    the number of lattice points in $\Phi^{-1}(B(0, r))$ is 
    \begin{equation} \label{eq:geometry2}
        \#(\Z^2 \cap r^{1/e}D_\Phi) = \vol(D_\Phi) r^{2/e} + O(r^{1/e}).
    \end{equation}
    where the $O$-term depends on $\Phi$.
    
    Heuristically, since 
    \[
        \#(\Z^2 \cap B(0, r)) = (2\lfloor r \rfloor + 1)^2 \ \text{ and } \ 
        \#(\Z^2_\vis \cap B(0, r)) = 2 \, \Num(\P^1(\Q), r),
    \]
    one expects that replacing $\Z^2$ by $\Z^2_\vis$ in \eqref{eq:geometry2}
    should introduce a factor of $6/\pi^2$ and an error of $O(r^{1/e} \log r)$.
    By the Chinese Remainder Theorem,
    further restricting $(a, b)$ 
    to satisfy $\gcd(f(a, b), g(a, b)) = d$
    should introduce a rational scale factor $C_d$ depending on $\Phi$, 
    seeing as 
    \[
    d \mid f(a, b) \text{ and } d \mid g(a, b) 
    \iff 
    (a, b) \in \pi_d^{-1}\big(V(f, g)(\Z/d\Z)\big)
    \]
    where $\pi_d : \Z^2_\vis \to \P^1(\Z/d\Z)$ is the reduction map
    to the projective line over the ring of integers mod $d$.
    Thus, 
    one expects that each summand in (*) 
    is asymptotically
    \begin{equation} \label{eq:geometry3}
        \frac{6}{\pi^2} C_d \vol(D_\Phi) d^{2/e} X^{2/e} + O(X^{1/e} \log X).
    \end{equation}
    
    As for the number of terms, the elimination property of the resultant implies that $d$ is a divisor of $R_\Phi := \Res(f, g)$.
    To wit,
    one can find (by inverting the Sylvester matrix) 
    homogeneous forms $A_1, \ldots, A_4$ in $\Z[f_e, g_e, \ldots, f_0, g_0][x, y]$ of degree $e - 1$
    such that 
    \[
        A_1 f + A_2 g = R_\Phi x^{2e-1} 
        \quad\text{and}\quad 
        A_3 f + A_4 g = R_\Phi y^{2e-1}.
    \]
    Setting $x = a$ and $y = b$,
    and noting that $v_p(A_i(a, b)) \ge 0$ (as $v_p(f, g) \ge 0$),
    yields
    \[
    \min\{v_p(f(a, b)), v_p(g(a, b))\}
    \le v_p(R_\Phi) + (2e - 1) \min\{v_p(a), v_p(b)\}
    \]
    for all primes $p$. 
    The l.h.s.~is $v_p(d)$ and the r.h.s.~is $v_p(R_\Phi)$,
    whence the claim.
    
    Summing \eqref{eq:geometry3} over all $d$ dividing $R_\Phi$,
    we are thus led to the heuristic formula
    \begin{equation} \label{eq:geometry4}
        \Num(\phi(\P^1(\Q)), X)
        = 
        \frac{3}{j \pi^2} \vol(D_\Phi) \!\! \sum_{\ d \mid R_\Phi} \! C_d \, d^{2/e} X^{2/e} + O(X^{1/e} \log X).
    \end{equation}
    For $\Phi(x, y) = (cy^2 - 4x^2, 4y^2)$ 
    we have $j = e = 2$, 
    $\vol(D_\Phi) = \frac{1}{2}\gamma(c)$, $R_\Phi = 4^4$, and
    \[
    C_d = \begin{cases} 
    2/3 & d = 1, \\
    1/6 & d = 4 \text{ or } 16
    \end{cases} 
    \]
    so that $C_1 + 4C_4 + 16 C_{16} = 4$.
\end{rmk}

\bibliographystyle{plain}
\bibliography{refs}

\end{document}